\documentclass[11pt,a4paper,british]{amsart}
\usepackage{amssymb}
\usepackage{xypic,bbm,mathrsfs,babel}
\usepackage[T1]{fontenc}
\usepackage{graphicx}
\usepackage[usenames,dvipsnames]{xcolor}
\definecolor{nred}{rgb}{0.7,0,0}
\definecolor{nblue}{rgb}{0,0,0.7}
\usepackage[ocgcolorlinks,citecolor=nblue,linkcolor=nblue,urlcolor=nred,backref=page]{hyperref}
\usepackage[capitalize,nameinlink]{cleveref}
\usepackage[marginratio=1:1,paperwidth=460pt,paperheight=732pt,height=632pt,width=360pt]{geometry}

\usepackage{microtype}
\usepackage{caption}
\usepackage{subcaption}

\makeatletter
\def\@cite#1#2{\textcolor{nblue}{[#1\if@tempswa , #2\fi]}}
\renewcommand*{\@seccntformat}[1]{%
   \csname the#1\endcsname.\hspace{1mm}} 
\makeatletter
\def\@settitle{\begin{center}%
  \baselineskip14\p@\relax
    \sc
 \Large\@title
 \end{center}%
}
\def\@setauthors{%
  \begingroup
  \def\thanks{\protect\thanks@warning}%
  \trivlist
  \centering\footnotesize \@topsep30\p@\relax
  \advance\@topsep by -\baselineskip
  \item\relax
  \author@andify\authors
  \def\\{\protect\linebreak}%
  \sc\large{\authors}%
  \ifx\@empty\contribs
  \else
    ,\penalty-3 \space \@setcontribs
    \@closetoccontribs
  \fi
  \endtrivlist
  \endgroup
}
\def\@settitle{\begin{center}%
  \baselineskip14\p@\relax
    \sc
 \Large\@title
 \end{center}%
}
\def\@setauthors{%
  \begingroup
  \def\thanks{\protect\thanks@warning}%
  \trivlist
  \centering\footnotesize \@topsep30\p@\relax
  \advance\@topsep by -\baselineskip
  \item\relax
  \author@andify\authors
  \def\\{\protect\linebreak}%
  \sc\large{\authors}%
  \ifx\@empty\contribs
  \else
    ,\penalty-3 \space \@setcontribs
    \@closetoccontribs
  \fi
  \endtrivlist
  \endgroup
}
\makeatother

\renewcommand*{\eqref}[1]{\textcolor{nblue}{(\ref{#1})}}

\numberwithin{equation}{section}

\newtheorem{theorem}{Theorem}[section]
\newtheorem{proposition}[theorem]{Proposition}
\newtheorem{lemma}[theorem]{Lemma}

\newtheorem{maintheorem}{Theorem}

\theoremstyle{definition}
\newtheorem{definition}[theorem]{Definition}
\theoremstyle{remark}
\newtheorem{remark}[theorem]{Remark}
\newtheorem{example}[theorem]{Example}

\newcommand*{\J}{\mathsf{J}_+(\mathbb{CP}(a_1,a_2))}
\newcommand*{\teiler}{c}
\newcommand*{\SU}{\mathrm{SU}(2)}
\newcommand*{\su}{\mathfrak{su}(2)}
\newcommand*{\U}{\mathrm{U}(1)}
\newcommand*{\varz}{z}
\newcommand*{\varw}{w}
\renewcommand*{\Re}{\mathrm{Re}}
\renewcommand*{\Im}{\mathrm{Im}}
\newcommand*{\ov}{\overline}	
\newcommand*{\R}{\mathbb{R}}
\newcommand*{\Z}{\mathbb{Z}}
\renewcommand*{\i}{\mathrm{i}}
\newcommand*{\C}{\mathbb{C}}
\newcommand*{\Co}{\mathbb{C}}

\newcommand*{\D}{\mathbb{D}}
\newcommand*{\Orb}{\mathcal{O}}
\newcommand*{\To}{\rightarrow}
\newcommand*{\MTo}{\mapsto}
\renewcommand*{\d}{\mathrm{d}}
\newcommand*{\area}{d\sigma_g}
\renewcommand*{\frame}{\mathscr{P}}
\newcommand*{\e}{\mathrm{e}}
\newcommand*{\un}{\underline}
\newcommand*{\dualcon}{Z}
\newcommand*{\con}{\zeta}
\newcommand*{\bihol}{\Xi}
\newcommand*{\linevert}{L}
\newcommand*{\linepath}{P}
\newcommand*{\esu}{\begin{pmatrix} \varz & -\ov\varw \\ \varw & \ov \varz \end{pmatrix}}
\newcommand*{\inc}{\hspace{3pt}\rule[0.5pt]{2mm}{0.5pt}\rule[0.5pt]{0.5pt}{4.5pt}\hspace{3pt}}

\subjclass[2010]{53B40, 53A20, 58C10}
\keywords{Finsler structures, constant curvature, Zoll metrics, spindle orbifolds, holomorphic curves}

\begin{document}
\author[C.~Lange]{Christian Lange}
\address{Christian Lange\newline\indent Mathematisches Institut der Universit\"at M\"unchen \newline\indent Theresienstr.~39, D-80333 Munich, Germany}
\email{lange@math.lmu.de, clange.math@gmail.com}
\author[T.~Mettler]{Thomas Mettler}
\address{Thomas Mettler\newline\indent Institut f\"ur Mathematik, Goethe-Universit\"at Frankfurt \newline\indent Robert-Mayer-Str.~10, D-60325 Frankfurt am Main, Germany}
\email{mettler@math.uni-frankfurt.de, mettler@math.ch}

\title[Veronese Embedding and Finsler 2-Spheres]{Deformations of the Veronese Embedding and Finsler 2-Spheres of Constant Curvature}

\date{13th March 2021}

\begin{abstract}
We establish a one-to-one correspondence between Finsler structures on the $2$-sphere with constant curvature $1$ and all geodesics closed on the one hand, and Weyl connections on certain spindle orbifolds whose symmetric Ricci curvature is positive definite and all of whose geodesics are closed on the other hand. As an application of our duality result, we show that suitable holomorphic deformations of the Veronese embedding $\mathbb{CP}(a_1,a_2)\To \mathbb{CP}(a_1,(a_1+a_2)/2,a_2)$ of weighted projective spaces provide examples of Finsler $2$-spheres of constant curvature whose geodesics are all closed. 
\end{abstract}

\maketitle

\section{Introduction}

\subsection{Background}
Riemannian metrics of constant curvature on closed surfaces are fully understood, a complete picture in the case of Finsler metrics is however still lacking. Akbar-Zadeh~\cite{MR1052466} proved a first key result by showing that on a closed surface a Finsler metric of constant negative curvature must be Riemannian, and locally Minkowskian in the case where the curvature vanishes identically (see also~\cite{MR1898191}). In the case of constant positive curvature a Finsler metric must still be Riemannian, provided it is reversible~\cite{MR2313331}, but the situation turns out to be much more flexible in the non-reversible case.

 Katok~\cite{MR0331425} gave the first examples (later analysed by Ziller~\cite{MR743032}) of non-re\-ver\-si\-ble Finsler metrics of constant positive curvature, though it was only realized later that Katok's examples actually have constant curvature. Meanwhile, Bryant~\cite{MR1403574} gave another construction of non-re\-ver\-si\-ble Finsler metrics of constant positive curvature on the $2$-sphere $S^2$ and in subsequent work~\cite{Bry} classified all Finsler metrics on $S^2$ having constant positive curvature and that are projectively flat. Bryant also observed that every Zoll metric on $S^2$ with positive Gauss curvature gives rise to a Finsler metric on $S^2$ with constant positive curvature~\cite{MR1898190}. Hence, already by the work of Zoll \cite{ZOLL} from the beginning of the 20th century, the moduli space of constant curvature Finsler metrics on $S^2$ is known to be infinite-dimensional. Its global structure is however not well understood.

\subsection{A duality result}
Recently in~\cite{MR4195750}, Bryant et.~al.~inter alia showed that a Finsler metric on $S^2$ with constant curvature $1$ either admits a Killing vector field, or has all of its geodesics closed. Moreover, in the first case all geodesics become closed, and even of the same length, after a suitable (invertible) Zermelo transformation. Hence, in this sense the assumption that all geodesics are closed is not a restriction. However, in the second case the geodesics can in general have different lengths, unlike the geodesics of the Finsler metrics that arise from Bryant's construction using Zoll metrics.

In this paper we generalise Bryant's observation about Zoll metrics to a one-to-one correspondence which covers all Finsler metrics on $S^2$ with constant curvature $1$ and all geodesics closed. The correspondence arises from the classical notion of duality for so-called path geometries.

An~\textit{oriented path geometry} on an oriented surface $M$ prescribes an oriented path $\gamma\subset M$ for every oriented direction in $TM$. This notion can be made precise by considering the bundle $\pi : \mathbb{S}(TM):=\left(TM\setminus\{0_M\}\right)/\R^+ \to M$ which comes equipped with a tautological co-orientable contact distribution $C$.
An oriented path geometry is a one-dimensional distribution $\linepath \to \mathbb{S}(TM)$ so that $\linepath$ together with the vertical distribution $\linevert=\ker \pi^{\prime}$ span $C$. 


The orientation of $M$ equips $P$ and $L$ with an orientation as well and following~\cite{Bry}, a $3$-manifold $N$ equipped with a pair of oriented one-di\-men\-sional distributions $(\linepath,\linevert)$ spanning a contact distribution is called an~\textit{oriented generalized path geometry}. In this setup the surface $M$ is replaced with the leaf space of the foliation $\mathcal{\linevert}$ defined by $\linevert$ and the leaf space of the foliation $\mathcal{\linepath}$ defined by $\linepath$ can be thought of as the space of oriented paths of the oriented generalized path geometry $(\linepath,\linevert)$. We may reverse the role of $\linepath$ and $\linevert$ and thus consider the~\textit{dual} $(-\linevert,-\linepath)$ of the oriented generalized path geometry $(\linepath,\linevert)$, where here the minus sign indicates reversing the orientation. 

The unit circle bundle $\Sigma \subset TM$ of a Finsler metric $F$ on an oriented surface $M$ 
naturally carries the structure of an oriented generalized path geometry $(\linepath,\linevert)$. 
In the case where all geodesics are closed, the dual of the path geometry arising from a Finsler metric on the $2$-sphere with constant positive curvature arises from a certain generalization of a Besse $2$-orbifold \cite{La16} with positive curvature. Here a $2$-orbifold is called \emph{Besse} if all its geodesics are closed. Namely, using the recent result~\cite{MR4195750} by Bryant et al.~about such Finsler metrics (see~\cref{thm:BFIVMZ} below), we show that the space of oriented geodesics is a spindle-orbifold $\Orb$ 
-- or equivalently, a~\textit{weighted projective line} -- which comes equipped with a positive Besse--Weyl structure. By this we mean an affine torsion-free connection $\nabla$ on $\Orb$ which preserves some conformal structure -- a so-called~\textit{Weyl connection} -- and which has the property that the image of every maximal geodesic of $\nabla$ is an immersed circle. Moreover, the symmetric part of the Ricci curvature of $\nabla$ is positive definite. Conversely, having such a positive Besse--Weyl structure on a spindle orbifold, we show that the dual path geometry yields a Finsler metric on $S^2$ with constant positive curvature all of whose geodesics are closed. More precisely, we prove the following duality result which generalizes~\cite[Theorem 3]{MR1898190} and \cite[Proposition 6, Corollary 2]{MR2313331} by Bryant:
\begin{maintheorem}\label{thm:duality}
There is a one-to-one correspondence between Finsler structures on $S^2$ with constant Finsler--Gauss curvature $1$ and all geodesics closed on the one hand, and positive Besse--Weyl structures on spindle orbifolds $S^2(a_1,a_2)$ with $\teiler:=\gcd(a_1,a_2)\in \{1,2\}$, $a_1\geqslant a_2$, $2|(a_1+a_2)$ and $\teiler^3|a_1a_2$ on the other hand. More precisely, 
\begin{enumerate}
\item such a Finsler metric with shortest closed geodesic of length $2\pi\ell \in (\pi,2\pi]$, $\ell=p/q\in (\frac{1}{2},1]$, $\gcd(p,q)=1$, gives rise to a positive Besse--Weyl structure on $S^2(a_1,a_2)$ with $a_1=q$ and $a_2=2p-q$, and
\item a positive Besse--Weyl structure on such a $S^2(a_1,a_2)$ gives rise to such a Finsler metric on $S^2$ with shortest closed geodesic of length $2\pi\left(\frac{a_1+a_2}{2a_1}\right) \in (\pi,2\pi]$,
\end{enumerate}
and these assignments are inverse to each other. Moreover, two such Finsler metrics are isometric if and only if the corresponding Besse--Weyl structures coincide up to a diffeomorphism.
\end{maintheorem}

\subsection{Construction of examples} 

In~\cite{MR3144212}, it is shown that Weyl connections with prescribed (unparametrised) geodesics on an oriented surface $M$ are in one-to-one correspondence with certain holomorphic curves into the ``twistor space'' over $M$. In \cref{sec:examples} we make use of this observation to construct deformations of positive Besse--Weyl structures on the weighted projective line $\mathbb{CP}(a_1,a_2)$ in a fixed projective class, by deforming the~\textit{Veronese embedding} of $\mathbb{CP}(a_1,a_2)$ into the weighted projective plane with weights $(a_1,(a_1+a_2)/2,a_2)$.
Applying our duality result, we obtain a corresponding real two-dimensional family of non-isometric, rotationally symmetric Finsler structures on the $2$-sphere with constant positive curvature and all geodesics closed, but not of the same length. The length of the shortest closed geodesic of the resulting Finsler metric is unchanged for our family of deformations and so it is of different nature than the Zermelo deformation used by Katok in the construction of his examples~\cite{MR0331425}. Moreover, we expect that not all of these examples are of Riemannian origin in the following sense (cf.~\cref{rem:projectively_equi_deformations}).
%


The construction of rotationally symmetric Zoll metrics on $S^2$ can be generalized to give an infinite-dimensional family of rotationally symmetric Riemannian metrics on spindle orbifolds all of whose geodesics are closed \cite{MR496885,La16}. Since every Levi-Civita connection is a Weyl connection, we obtain an infinite-dimensional family of rotationally symmetric positive Riemannian Besse--Weyl structures.

Furthermore, in~\cite{MR1979367,MR2747436} LeBrun--Mason construct a Weyl connection $\nabla$ on the $2$-sphere $S^2$ for every totally real embedding of $\mathbb{RP}^2$ into $\mathbb{CP}^2$ which is sufficiently close to the standard real linear embedding. The Weyl connection has the property that all of its maximal geodesics are embedded circles and hence defines a Besse--Weyl structure. In addition, they show that every such Weyl connection on $S^2$ is part of a complex $5$-dimensional family of Weyl connections having the same unparametrised geodesics (see also~\cite{MR3384876}). In particular, the Weyl connections of LeBrun--Mason that arise from an embedding of $\mathbb{RP}^2$ that is sufficiently close to the standard embedding provide examples of positive Besse--Weyl structures. The corresponding dual Finsler metrics on $S^2$ will have geodesics that are all closed and of the same length. 

A complete local picture of the space of Finsler $2$-spheres of constant positive curvature and with all geodesics closed likely requires extending the work of LeBrun--Mason to the orbifold setting. Our results in \cref{sec:examples} lay the foundation for such an extension. We hope to be able to built upon it in future work. 

\subsection*{Acknowledgements} The authors would like to thank Vladimir Matveev for drawing their attention to the work of Bryant and its possible extension as well as for helpful correspondence. CL was partially supported by the DFG funded project SFB/TRR 191. A part of the research for this article was carried out while TM was visiting FIM at ETH Z\"urich and Mathematisches Institut der Universit\"at zu K\"oln. TM thanks FIM for its hospitality and the DFG for travel support via the project SFB/TRR 191. TM was partially funded by the priority programme SPP 2026 ``Geometry at Infinity'' of DFG.

\section{Preliminaries}
\setcounter{maintheorem}{0}
\subsection{Background on orbifolds}\label{sub:background_orbifolds}
For a detailed account on different perspectives on orbifold we refer the reader to e.g. \cite{MR2359514,MR1744486,La18,MR0095520}. Here we only quickly recall some basic notions which are relevant for our purpose. An \emph{$n$-dimensional Riemannian orbifold} $\Orb^n$ can be defined as a length space such that for each point $x \in \Orb$ there exists a neighbourhood $U$ of $x$ in $\Orb$, an $n$-dimensional Riemannian manifold $M$ and a finite group $\Gamma$ acting by isometries on $M$ such that $U$ and $M/\Gamma$ are isometric \cite{La18}. In this case we call $M$ a manifold chart for $\Orb$. Every Riemannian orbifold admits a canonical smooth structure, i.e., roughly speaking, there exist equivariant, smooth transition maps between manifolds charts. Conversely, every smooth orbifold is ``metrizable'' in the above sense. For a point $x$ on an orbifold the linearised isotropy group of a preimage of $x$ in a manifold chart is uniquely determined up to conjugation. Its conjugacy class is denoted as $\Gamma_x$ and is called the \emph{local group} of $\Orb$ at $x$. A point $x\in \Orb$ is called \emph{regular} if its local group is trivial and otherwise \emph{singular}.

For example, the metric quotient $\Orb_{a}$, $a=(a_1,a_2)$, of the unit sphere $S^{3}\subset \C^{2}$ by the isometric action of $S^1\subset \C$ defined by
\[
				z (z_1,z_2)=(z^{a_1}z_1,z^{a_2}z_2)
\]
for co-prime numbers $a_1\geqslant a_2$ is a Riemannian orbifold which is topologically a $2$-sphere, but which metrically has two isolated singular points with cyclic local groups of order $a_1$ and $a_2$. We denote the underlying smooth orbifold as $S^2(a_1,a_2)$ and refer to it as a $(a_1,a_2)$-\emph{spindle orbifold}. The quotient map $\pi$ from $S^3$ to $\Orb_{a}$ is an example of an \emph{orbifold (Riemannian) submersion}, in the sense that for every point $z$ in $S^3$, there is a neighbourhood $V$ of $z$ such that $M/\Gamma=U=\pi(V)$ is a chart, and $\pi|_V$ factors as $V\stackrel{\tilde{\pi}}{\longrightarrow} M{\longrightarrow} M/\Gamma=U$, where $\tilde{\pi}$ is a standard submersion. The anti-Hopf action of $S^1$ on $S^3$ defined by $z (z_1,z_2)=(zz_1,z^{-1}z_2)$ commutes with the above $S^1$-action and induces an isometric $S^1$-action on $\Orb_{a}$. Let $\Gamma_k$ be a cyclic subgroup of the anti-Hopf $S^1$-action. The quotient $S^3/\Gamma_k$ is a \emph{lens space} of type $L(k,1)$. By moding out such $\Gamma_k$-actions on $\Orb_{a}$ we obtain spindle orbifolds $S^2(a_1,a_2)$ with arbitrary $a_1$ and $a_2$ as quotients. These spaces fit in the following commutative diagram
\[
	\begin{xy}
		\xymatrix
		{
			  S^3  \ar[r] \ar[d] & \Orb_{a} \cong S^2(a_1,a_2) \ar[d]  \\
		  S^3/\Gamma_k \cong L(k,1)  \ar[r] & \Orb_{a}/\Gamma_k \cong S^2(k'a_1,k'a_2)  
		}
	\end{xy}
\]
for some $k'|k$. Here the left vertical map is an example of a (Riemannian) orbifold covering $p:\Orb \To \Orb'$, i.e. each point $x\in \Orb'$ has a neighbourhood $U$ isomorphic to some $M/\Gamma$ for which each connected component $U_i$ of $p^{-1}(U)$ is isomorphic to $M/\Gamma_i$ for some subgroup $\Gamma_i<\Gamma$ such that the isomorphisms are compatible with the natural projections $M/\Gamma_i \To M/\Gamma$ (see \cite{La18} for a metric definition). Thurston has shown that the theory of orbifold coverings works analogously to the theory of ordinary coverings \cite{Thurston}. In particular, there exist universal coverings and one can define the \emph{orbifold fundamental group} $\pi_1^{orb}(\Orb)$ of a connected orbifold $\Orb$ as the deck transformation group of the universal covering. For instance, the orbifold fundamental group of $S^2(a_1,a_2)$ is a cyclic group of order $\gcd(a_1,a_2)$. Moreover, the number $k'$ in the diagram is determined in \cite[Theorem~4.10]{GL16} to be
\begin{equation}\label{eq:seifert_orbifold_relation}
k'=\frac{k}{\gcd(k,a_1-a_2)}.
\end{equation}

More generally, in his fundamental monograph \cite{MR1555366} Seifert studies foliations of $3$-manifolds by circles that are locally orbits of effective circle actions without fixed points (for a modern account see e.g. \cite{MR705527}). The orbit space of such a \emph{Seifert fibration} naturally carries the structure of a $2$-orbifold with isolated singularities. If both the $3$-manifold and the orbit space are orientable, then the Seifert fibration can globally be described as a decomposition into orbits of an effective circle action without fixed points (see e.g. \cite[Section~2.4]{La16} and the references stated therein). In particular, in \cite[Chapter~11]{MR1555366} Seifert shows that any Seifert fibration of the $3$-sphere is given by the orbit decomposition of a weighted Hopf action. The classification of Seifert fibrations of lens spaces, their quotients and their behaviour under coverings is described in detail in \cite{GL16}. Let us record the following special statement which will be needed later.
\begin{lemma}\label{lem:seifert_on_rp3}Let $\mathcal{F}$ be a Seifert fibration of $\mathbb{RP}^3\cong L(2,1)$ with orientable quotient orbifold. Then the quotient orbifold is a $S^2(a_1,a_2)$ spindle orbifold, $a_1\geqslant a_2$, with $2|(a_1+a_2)$, $\teiler:=\gcd(a_1,a_2) \in \{1,2\}$ and $\teiler^3|a_1a_2$.
\end{lemma}
\begin{proof} Since $\mathbb{RP}^3$ and the quotient surface are orientable, the Seifert fibration is induced by an effective circle action without fixed points. It follows from the homotopy sequence, that the orbifold fundamental group of the quotient is either trivial or $\Z_2$ \cite[Lemma~3.2]{MR705527}. In particular, the quotient has to be a spindle orbifold (see e.g. \cite[Chapter~3]{MR705527} or \cite[Chapter~10]{MR1555366}). Moreover, such a Seifert fibration is covered by a Seifert fibration of $S^3$ \cite[Theorem~5.1]{GL16} with quotient $S^2(a^0_1,a^0_2)$ for co-prime $a^0_1$ and $a^0_2$ with $a_i=aa^0_i$, and with
\[
		a=\frac{2}{\gcd(2,a^0_1+a^0_2)}=\frac{2}{\gcd(2,a^0_1-a^0_2)}
\]
by \cite[Theorem~4.10]{GL16}. This implies $2|(a_1+a_2)$, $\teiler:=\gcd(a_1,a_2) \in \{1,2\}$ and $\teiler^3|a_1a_2$ as claimed.
\end{proof}

Usually notions that make sense for manifolds can also be defined for orbifolds. The general philosophy is to either define them in manifold charts and demand them to be invariant under the action of the local groups (and transitions between charts as in the manifold case) like in the case of a Riemannian metric, or to demand certain lifting conditions. For instance, a map between orbifolds is called \emph{smooth} if it locally lifts to smooth maps between manifolds charts. 
Let us also explicitly mention  that the tangent bundle of an orbifold can be defined by gluing together quotients of the tangent bundles of manifold charts by the actions of local groups \cite[Proposition~1.21]{MR2359514}. In particular, if the orbifold has only isolated singularities, then its unit tangent bundle (with respect to any Riemannian metric) is in fact a manifold. For instance, the unit tangent bundle of a $S^2(a_1,a_2)$ spindle orbifold is an $L(a_1+a_2,1)$ lens space \cite[Lemma~3.1]{La16}. General vector bundles on orbifolds can be similarly defined on the level of charts. We will only work with vector bundles on spindle orbifolds $S^2(a_1,a_2)$ which can be described as associated bundles $\SU\times_{S^1} V$ for some linear representation of $S^1$ on a vector space $V$.

In the sequel we liberally use orbifold notions which follow this general philosophy without further explanation, and refer to the literature for more details.

\subsection{Besse orbifolds}\label{sub:besse_orbifolds} The Riemannian spindle orbifolds $\Orb_a\cong S^2(a_1,a_2)$ constructed in the preceding section have the additional property that all their geo\-de\-sics are closed, i.e. any geodesic factors through a closed geodesic. Here an \emph{(orbifold) geodesic} on a Riemannian orbifold is a path that can locally be lifted to a geodesic in a manifold chart, and a \emph{closed geodesic} is a loop that is a geodesic on each subinterval. We call a Riemannian metric on an orbifold as well as a Riemannian orbifold \emph{Besse}, if all its geodesics are closed. The moduli space of (rotationally symmetric) Besse metrics on spindle orbifolds is infinite-dimensional \cite{MR496885,La16}. For more details on Besse orbifolds we refer to \cite{ALR,La16}.

\subsection{Finsler structures.} 
A Finsler metric on a manifold is -- roughly speaking -- a Banach norm on each tangent space varying smoothly from point to point. Instead of specifying the family of Banach norms, one can also specify the norm's unit vectors in each tangent space. Here we only consider oriented Finsler surfaces and use definitions for Finsler structures from~\cite{Bry}:

A~\textit{Finsler structure} on an oriented surface $M$ is a smooth hypersurface $\Sigma \subset TM$ for which the basepoint projection $\pi : \Sigma \to M$ is a surjective submersion which has the property that for each $p \in M$ the fibre $\Sigma_p=\pi^{-1}(p)=\Sigma \cap T_pM$ is a closed, strictly convex curve enclosing the origin $0\in T_pM.$ A smooth curve $\gamma : [a,b] \to M$ is said to be a $\Sigma$-curve if its velocity $\dot{\gamma}(t)$ lies in $\Sigma$ for every time $t \in [a,b]$. For every immersed curve $\gamma : [a,b] \to M$ there exists a unique orientation preserving diffeomorphism $\Phi : [0,\mathscr{L}] \to [a,b]$ such that $\phi:=\gamma \circ \Phi$ is a $\Sigma$-curve. The number $\mathscr{L} \in \R^{+}$ is the length of $\gamma$ and the curve $\dot{\phi} : [a,b] \to \Sigma$ is called the \textit{tangential lift} of $\gamma$. Note that in general the length may depend on the orientation of the curve. 

Cartan~\cite{cartan1930} has shown how to associate a coframing to a Finsler structure on an oriented surface $M$. For a modern reference for Cartan's construction the reader may consult~\cite{bryantprescribedfinslercurvature}. Let $\Sigma \subset TM$ be a Finsler structure. Then there exists a unique coframing $\frame=(\chi,\eta,\nu)$ of $\Sigma$ with dual vector fields $(X,H,V)$ which satisfies the structure equations
\begin{align}\label{eq:struceqfinsler}
\begin{split}
\d\chi=&-\eta\wedge\nu,\\
\d\eta=&-\nu\wedge(\chi-I\eta),\\
\d\nu=&-(K\chi-J\nu)\wedge\eta,
\end{split}
\end{align}
for some smooth functions $I,J,K :\Sigma \to \R$. Moreover the $\pi$-pullback of any positive volume form on $M$ is a positive multiple of $\chi\wedge\eta$ and the tangential lift of any $\Sigma$-curve $\gamma$ satisfies 
\[
\dot{\gamma}^*\eta=0 \quad \text{and} \quad \dot{\gamma}^*\chi=\d t. 
\] 
A $\Sigma$-curve $\gamma$ is a $\Sigma$-\textit{geodesic}, that is,~a critical point of the length functional, if and only if its tangential lift satisfies $\dot{\gamma}^*\nu=0$. The integral curves of $X$ therefore project to $\Sigma$-geodesics on $M$ and hence the flow of $X$ is called the \textit{geodesic flow} of $\Sigma$. 

For a \textit{Riemannian Finsler structure} the functions $I,J$ vanish identically, as a result of which $K$ is constant on the fibres of $\pi : \Sigma\to M$ and therefore the $\pi$-pullback of a function on $M$ which is the Gauss curvature $K_g$ of $g$. Since in the Riemannian case the function $K$ is simply the Gauss curvature, it is usually called the \textit{Finsler--Gauss curvature}. In general $K$ need not be constant on the fibres of $\pi : \Sigma \to M$.

Let $\Sigma \subset TM$ and $\hat{\Sigma}\subset T\hat{M}$ be two Finsler structures on oriented surfaces with coframings $\frame$ and $\hat{\frame}$. An orientation preserving diffeomorphism $\Phi : M \to \hat{M}$ with $\Phi^{\prime}(\Sigma)=\hat{\Sigma}$ is called a \textit{Finsler isometry}. It follows that for a Finsler isometry $(\Phi^{\prime}\vert_{\Sigma})^*\hat{\frame}=\frame$ and conversely any diffeomorphism $\Xi : \Sigma \to \hat{\Sigma}$ which pulls-back $\hat{\frame}$ to $\frame$ is of the form $\Xi=\Phi^{\prime}$ for some Finsler isometry $\Phi : M \to \hat{M}$.

Following~\cite[Def.~1]{Bry}, we use the following definition:
\begin{definition}
A coframing $(\chi,\eta,\nu)$ on a $3$-manifold $\Sigma$ satisfying the structure equations~\eqref{eq:struceqfinsler} for some functions $I,J$ and $K$ on $\Sigma$ will be called a~\textit{generalized Finsler structure}. 
\end{definition}

As in the case of a Finsler structure we denote the dual vector fields of $(\chi,\eta,\nu)$ by $(X,H,V)$. Note that a generalized Finsler structure naturally defines an oriented generalized path geometry by defining $P$ to be spanned by $X$ while calling positive multiples of $X$~\textit{positive} and by defining $L$ to be spanned by $V$ while calling positive multiples of $V$~\textit{positive}. 

\begin{example}\label{riem}
Let $(\Orb,g)$ be an oriented Riemannian $2$-orbifold. In particular, $\Orb$ has only isolated singularities. Then the \textit{unit tangent bundle} 
\[
S\Orb:=\left\{v \in T\Orb\, :\,\vert v \vert_g=1\right\}\subset T\Orb
\]
is a manifold, and like in the case of a smooth Finsler structure it can be equipped with a canonical coframing as well. In order to distinguish the Riemannian orbifold case from the smooth Finsler case, we will use the notation $(\alpha,\beta,\con)$ instead of $(\chi,\eta,\nu)$ for the coframing. The construction is as follows: A manifold chart $M/\Gamma$ of $\Orb$ gives rise to a manifold chart $SM/\Gamma$ of $S\Orb$. In such a chart the first two coframing forms are explicitly given by 
\[
\alpha_v(w):=g(\pi^{\prime}_v(w),v), \quad \beta_v(w):=g(\pi^{\prime}_v(w),iv), \quad w \in T_vSM. 
\]
Here $\pi : S\Orb \to \Orb$ denotes the basepoint projection and $i : TM \to TM$ the rotation of tangent vectors by $\pi/2$ in positive direction. Note that these expressions are invariant under the group action of $\Gamma$ and hence in fact define forms on $\Orb M$. The third coframe form $\con$ is the \textit{Levi-Civita connection form} of $g$ and we have the structure equations
\[
\d\alpha=-\beta\wedge \con, \qquad \d\beta=-\con\wedge\alpha,\qquad \d\con=-(K_g\circ \pi) \alpha\wedge\beta,
\]
where $K_g$ denotes the Gauss curvature of $g$. Moreover, note that $\pi^*\area=\alpha\wedge\beta$ where $\area$ denotes the area form of $\Orb$ with respect to $g$. Denoting the vector fields dual to $(\alpha,\beta,\con)$ by $(A,B,\dualcon)$ we observe that the flow of $\dualcon$ is $2\pi$-periodic. Finally, if $\Orb$ is a manifold, then the coframing $(\alpha,\beta,\con)$ agrees with Cartan's coframing $(\chi,\eta,\nu)$ on the Riemannian Finsler structure $\Sigma=S\Orb$.
\end{example}

\subsection{Weyl structures and connections}

A~\textit{Weyl connection} on an orbifold $\Orb$ is an affine torsion-free connection on $\Orb$ preserving some conformal structure $[g]$ on $\Orb$ in the sense that its parallel transport maps are angle preserving with respect to $[g]$. An affine torsion-free connection $\nabla$ is a Weyl connection with respect to the conformal structure $[g]$ on $\Orb$ if for some (and hence any) conformal metric $g \in [g]$ there exists a $1$-form $\theta \in \Omega^1(\Orb)$ such that 
\begin{equation}\label{eq:defpropweylcon}
\nabla g=2\theta\otimes g. 
\end{equation}
Conversely, it follows from Koszul's identity that for every pair $(g,\theta)$ consisting of a Riemannian metric $g$ and $1$-form $\theta$ on $\Orb$ the connection 
\begin{equation}\label{eq:exprweylcon}
{}^{(g,\theta)}\nabla_XY={}^g\nabla_XY+g(X,Y)\theta^{\sharp}-\theta(X)Y-\theta(Y)X, \quad X,Y \in \Gamma(T\Orb) 
\end{equation}
is the unique affine torsion-free connection satisfying~\eqref{eq:defpropweylcon}. Here ${}^g\nabla$ denotes the Levi-Civita connection of $g$ and $\theta^{\sharp}$ is the vector field dual to $\theta$ with respect to $g$. Notice that for $u \in C^{\infty}(\Orb)$ we have the formula
\[
{}^{\exp(2u)g}\nabla_XY={}^g\nabla_XY-g(X,Y)(\d u)^{\sharp}+\d u(X)Y+\d u(Y)X, \quad X,Y \in \Gamma(T\Orb).
\]
From which one easily computes the identity
\[
{}^{(\exp(2u)g,\theta+\d u)}\nabla={}^{(g,\theta)}\nabla.
\]
Consequently, we define a~\textit{Weyl structure} to be an equivalence class $[(g,\theta)]$ subject to the equivalence relation
\[
(\hat{g},\hat{\theta})\sim (g,\theta) \quad \iff \quad \hat{g}=\e^{2u}g\;\; \text{and} \;\; \hat{\theta}=\theta+\d u, \quad u \in C^{\infty}(\Orb). 
\]
Clearly, the mapping which assigns to a Weyl structure $[(g,\theta)]$ its Weyl connection ${}^{(g,\theta)}\nabla$ is a one-to-one correspondence between the set of Weyl structures -- and the set of Weyl connections on $\Orb$. 

The Ricci curvature of a Weyl connection ${}^{(g,\theta)}\nabla$ on $\Orb$ is
\[
\mathrm{Ric}\left({}^{(g,\theta)}\nabla\right)=\left(K_g-\delta_g\theta\right)g+2\d\theta
\]
where $\delta_g$ denotes the co-differential with respect to $g$. 

\begin{definition} We call a Weyl structure $[(g,\theta)]$~\textit{positive} if the symmetric part of the Ricci curvature of its associated Weyl connection is positive definite.
\end{definition}
In the case where $\Orb$ is oriented we may equivalently say the Weyl structure $[(g,\theta)]$ is positive if the $2$-form $(K_g-\delta_g\theta)\area$ -- which only depends on the orientation and given Weyl structure -- is an orientation compatible volume form on $\Orb$. 
Note that by the Gauss--Bonnet theorem \cite{MR0095520} simply connected spindle orbifolds are the only simply connected $2$-orbifolds carrying positive Weyl structures. 

We now obtain:
\begin{lemma}
Every positive Weyl structure contains a unique pair $(g,\theta)$ satisfying $K_g-\delta_g\theta=1$. 
\end{lemma}
\begin{proof}
We have the following standard identity for the change of the Gauss cur\-va\-ture under conformal change
\[
K_{\e^{2u}g}=\e^{-2u}\left(K_g-\Delta_g u\right)
\]
where $\Delta_g=-(\d\delta_g+\delta_g\d)$ is the negative of the Laplace--de Rham operator. Also, we have the identity
\[
\delta_{e^{2u}g}=\e^{-2u}\delta_g
\]
for the co-differential acting on $1$-forms. 

If $[(g,\theta)]$ is a positive Weyl structure, we may take any representative $(g,\theta)$, define $u=\frac{1}{2}\ln(K_g-\delta_g\theta)$ and consider the representative $(\hat{g},\hat{\theta})=(\e^{2u}g,\theta+\d u)$. Then we have
\[
K_{\hat{g}}-\delta_{\hat{g}}\hat{\theta}=\frac{K_g+\delta_g\d u}{K_g-\delta_g\theta}-\frac{\delta_g(\theta+\d u)}{K_g-\delta_g\theta}=1.
\]
Suppose the two representative pairs $(g,\theta)$ and $(\hat{g},\hat{\theta})$ both satisfy $K_{\hat{g}}-\delta_{\hat{g}}\hat{\theta}=K_g-\delta_g\theta=1$. Since they define the same Weyl connection, the expression for the Ricci curvature implies that
$\hat{g}=(K_{\hat{g}}-\delta_{\hat{g}}\hat{\theta})\hat{g}=(K_g-\delta_g\theta)g=g$ and hence also $\hat{\theta}=\theta$, as claimed. 
\end{proof}
\begin{definition}
For a positive Weyl structure $[(g,\theta)]$ we call the unique representative pair $(g,\theta)$ satisfying $K_g-\delta_g\theta=1$ the~\textit{natural gauge} of $[(g,\theta)]$. 
\end{definition}
\begin{lemma}\label{lemma:posweyltofinsler}
Let $[(g,\theta)]$ be a positive Weyl structure on an orientable $2$-orbifold $\Orb$ with natural gauge $(g,\theta)$ and let $\pi : S\Orb \to \Orb$ denote the unit tangent bundle of $g$ equipped with its canonical coframing $(\alpha,\beta,\con)$. Then the coframing
$$
\chi:=\pi^*(\star_g\theta)-\con, \quad \eta:=-\beta, \quad \nu:=-\alpha
$$
defines a generalized Finsler structure of constant Finsler--Gauss curvature $K=1$ on $S\Orb$. 
\end{lemma}
\begin{proof}
We compute that 
\[
\d\chi=\d\left(\pi^*(\star_g\theta)-\con\right)=\pi^*\left((K_g-\delta_g\theta)\area\right)=\alpha\wedge\beta=-\eta\wedge\nu 
\]
and
\[
\d\eta=-\d\beta=\con\wedge\alpha=(\chi-\pi^*(\star_g\theta))\wedge\nu=-\nu\wedge(\chi-\pi^*(\star_g\theta))
\]
Now observe that $\pi^*(\star_g\theta)=-\dualcon(\theta)\alpha+\theta\beta$ where on the right hand side we think of $\theta$ as a real-valued function on $S\Orb$. Since $\nu=-\alpha$, we thus have
\[
\d\eta=-\nu\wedge\left(\chi-I\eta\right),
\]
for $I=-\theta$, again interpreted as a function on $S\Orb$. Likewise, we obtain
\[
\d\nu=-\d\alpha=\beta\wedge\con=-(\chi-\pi^*(\star_g\theta))\wedge\eta=-\left(\chi-J\nu\right)\wedge\eta, 
\]
where $J=\dualcon(\theta)$. The claim follows.
\end{proof}
\begin{remark}
We remark that correspondingly we have a natural gauge $(g,\theta)$ for a negative Weyl structure, that is, $(g,\theta)$ satisfy $K_g-\delta_g\theta=-1$. On a closed oriented surface (necessarily of negative Euler characteristic) the associated flow generated by the vector field $A-\dualcon(\theta)\dualcon$ falls into the family of flows introduced in~\cite{MR3968880}. In particular, its dynamics is Anosov. 
\end{remark}

The geometric significance of the form $\chi$ in~\cref{lemma:posweyltofinsler} is described in the following statement. For a proof in the manifold case -- which carries over mutatis mutandis to the orbifold case -- the reader may consult~\cite[Lemma 3.1]{MR4109900}. 
\begin{lemma}
Let $(g,\theta)$ be a pair of a Riemannian metric and a $1$-form on an orientable $2$-orbifold $\Orb$ and let $\pi : S\Orb \to \Orb$ denote the unit tangent bundle of $g$ with canonical coframing $(\alpha,\beta,\con)$. Then the leaves of the foliation defined by $\{\beta,\con-\pi^*(\star_g\theta)\}^{\perp}$ project to $\Orb$ to become the (unparametrised) oriented geodesics of the Weyl connection defined by $[(g,\theta)]$. 
\end{lemma}
We conclude this section with a definition:
\begin{definition}
An affine torsion-free connection $\nabla$ on $\Orb$ is called~\textit{Besse} if the image of every maximal geodesic of $\nabla$ is an immersed circle. A Weyl structure whose Weyl connection is Besse will be called a~\textit{Besse--Weyl structure}.
\end{definition}
Note that the Levi-Civita connection of any (orientable) Besse orbifold $\Orb$ (see \cref{sub:besse_orbifolds}) gives rise to a Besse--Weyl structure on $S \Orb$.
 
\section{A Duality Theorem}

Let us cite the following result from \cite{MR4195750}:

\begin{theorem}[Bryant, Foulon, Ivanov, Matveev, Ziller]\label{thm:BFIVMZ} Let $\Sigma \subset TS^2$ be a Finsler structure on $S^2$ with constant Finsler--Gauss curvature $1$ and all geodesics closed. Then there exists a shortest closed geodesic of length $2\pi \ell \in (\pi,2\pi]$ and the following holds:
\begin{enumerate}
\item Either $\ell=1$ and all geodesics have the same length $2 \pi$,
\item or $\ell=p/q \in (\frac{1}{2},1)$ with $p,q \in \mathbb{N}$ and $\gcd(p,q)=1$, and in this case all unit-speed geodesics have a common period $2\pi p$. Furthermore, there exists at most two closed geodesics with length less than $2\pi p$. A second one exists only if $2p-q>1$, and its length is $2\pi p/(2p-q) \in (2\pi,2p\pi)$.
\end{enumerate}
\end{theorem}
In particular, if all geo\-de\-sics of a Finsler metric on $S^2$ are closed, then its geo\-de\-sic flow is periodic with period $2\pi p$ for some integer $p$. 

We now have our main duality result:
\begin{maintheorem}
There is a one-to-one correspondence between Finsler structures on $S^2$ with constant Finsler--Gauss curvature $1$ and all geodesics closed on the one hand, and positive Besse--Weyl structures on spindle orbifolds $S^2(a_1,a_2)$ with $\teiler:=\gcd(a_1,a_2)\in \{1,2\}$, $a_1\geqslant a_2$, $2|(a_1+a_2)$ and $\teiler^3|a_1a_2$ on the other hand. More precisely, 
\begin{enumerate}
\item such a Finsler metric with shortest closed geodesic of length $2\pi\ell \in (\pi,2\pi]$, $\ell=p/q\in (\frac{1}{2},1]$, $\gcd(p,q)=1$, gives rise to a positive Besse--Weyl structure on $S^2(a_1,a_2)$ with $a_1=q$ and $a_2=2p-q$, and
\item a positive Besse--Weyl structure on such a $S^2(a_1,a_2)$ gives rise to such a Finsler metric on $S^2$ with shortest closed geodesic of length $2\pi\left(\frac{a_1+a_2}{2a_1}\right) \in (\pi,2\pi]$,
\end{enumerate}
and these assignments are inverse to each other. Moreover, two such Finsler metrics are isometric if and only if the corresponding Besse--Weyl structures coincide up to a diffeomorphism. 
\end{maintheorem}
\begin{proof}
In case of $2\pi$-periodic geodesic flows the first statement is already contained in~\cite{MR2313331}. To prove the general statement let $\Sigma \subset TS^2$ be a $K=1$ Finsler structure with $2\pi p$-periodic geodesic flow $\phi : \Sigma \times \R \to \Sigma$, i.e. the flow factorizes through a smooth, almost free $S^1$-action $\phi : \Sigma \times S^1 \to \Sigma$. The Cartan coframe will be denoted by $(\chi,\eta,\nu)$ and the dual vector fields by $(X,H,V)$. Since $\Sigma\cong SO(3)$ is an $L(2,1)$ lens space, the quotient map $\lambda$ for the $S^1$-action is a smooth orbifold submersion onto a spindle orbifold $\mathcal{O}=S^2(a_1,a_2)$, with $a_1\geqslant a_2$, $2|(a_1+a_2)$, $\teiler:=\gcd(a_1,a_2) \in \{1,2\}$ and $\teiler^3|a_1a_2$ by~\cref{lem:seifert_on_rp3}. With~\cref{thm:BFIVMZ} we see that $a_1=q$ and $a_2=2p-q$.
Since $X\inc\eta=X\inc\nu=0$, the $1$-forms $\eta$ and $\nu$ are semibasic for the projection $\lambda$ and using the structure equations for the $K=1$ Finsler structure, we compute the Lie derivative
\[
\mathrm{L}_{X}\left(\eta\otimes\eta+\nu\otimes\nu\right)=\nu\otimes \eta+\eta\otimes \nu-\eta\otimes \nu-\nu\otimes \eta=0. 
\]
Likewise, we compute $\mathrm{L}_X\left(\nu\wedge\eta\right)=0$. Hence the symmetric $2$-tensor $\eta\otimes\eta+\nu\otimes\nu$ and the $2$-form $\nu\wedge\eta$ are invariant under $\phi$ and therefore there exists a unique Riemannian metric $g$ on $\Orb$ for which $\lambda^*g=\eta\otimes\eta+\nu\otimes\nu$ where $\lambda : \Sigma \to \Orb$ is the natural projection. We may orient $\Orb$ in such a way that the pullback of the area form $\area$ of $g$ satisfies $\lambda^*\area=\nu\wedge\eta$. The structure equations also imply that $\chi,\eta,\nu$ are invariant under $(\phi_{2\pi})^{\prime}$ (cf. \cite[p.~186]{Bry}). Therefore the map
\[
\aligned
\Phi : \Sigma &\to T\Orb\\
v & \mapsto -\lambda^{\prime}_v\left(V(v)\right)\\
\endaligned
\]
and the forms $\chi,\eta,\nu$ are invariant under the action of the cyclic subgroup $\Gamma<S^1$ of order $p$ on $\Sigma$. Hence $\Phi$ factors through a map $\bar \Phi: \Sigma/\Gamma \rightarrow T\Orb$, and $\chi,\eta,\nu$ descend to $\Sigma/\Gamma$ where they define a generalized Finsler structure. The composition of $\bar\Phi$ with the canonical projection onto the projective sphere bundle $\mathbb{S}T\Orb:=\left(T\Orb\setminus\{0\}\right)/\R^{+}$ will be denoted by $\tilde{\Phi}$. Note that $\tilde{\Phi}$ is an immersion, thus a local diffeomorphism and by compactness of $\Sigma/\Gamma$ and connectedness of $\mathbb{S}T\Orb$ a covering map. Since by \cite[Lemma~3.1]{La16} both $\Sigma/\Gamma$ and $\mathbb{S}T\Orb$ have fundamental group of order $2p$, it follows that $\tilde{\Phi}$ is a diffeomorphism. Therefore, $\bar\Phi$ is an embedding which sends $\Sigma/\Gamma$ to the total space of the unit tangent bundle  $\pi : S\Orb\to\Orb$ of $g$. Abusing notation, we also write $\chi,\eta,\nu \in \Omega^1(S\Orb)$ to denote the pushforward with respect to $\bar\Phi$ of the Cartan coframe on $\Sigma/\Gamma$ . Also we let $\alpha,\beta,\con \in \Omega^1(S\Orb)$ denote the canonical coframe of $S\Orb$ with respect to the orientation induced by $\area$. More precisely, the pullback of $g$ to $S\Orb$ is $\alpha\otimes\alpha+\beta\otimes \beta$ and $\con$ denotes the Levi-Civita connection form. By construction, the map $\Phi$ sends lifts of $\Sigma$ geodesics onto the fibres of the projection $\pi : S\Orb \to \Orb$. Moreover, for $v\in \Sigma$ the projection $(\pi \circ\Phi)^*V(v)$ to $T_{(\pi\circ \Phi)(v)}\Orb$, i.e. the horizontal component of $\Phi^*V(v)$, is parallel to $\Phi(v)$ and so the vertical vector field $V$ on $\Sigma$ is mapped into the contact distribution defined by the kernel of $\beta$. Therefore, we see that $\beta$ and $\eta$ are linearly dependent and that $\nu(\Phi^* V),\alpha(\Phi^* V)<0$. However, since both $(\alpha,\beta)$ and $(\nu,\eta)$ are oriented orthonormal coframes for $g$, it follows that $\beta=-\eta$ and $\alpha=-\nu$. The structure equations for the coframing $(\alpha,\beta,\con)$ imply
\[
0=\d\alpha+\beta\wedge\con=-\d\nu-\eta\wedge\con=(\chi-J\nu)\wedge\eta-\eta\wedge\con=-\eta\wedge\left(\chi-J\nu+\con\right)
\]
and
\[
0=\d\beta+\con\wedge\alpha=-\d\eta-\con\wedge\nu=\nu\wedge\left(\chi-I\eta\right)-\con\wedge\nu=-\left(\chi-I\eta+\con\right)\wedge\nu,
\]
where again we abuse notation by also writing $I$ and $J$ for the pushforward of the functions $I$ and $J$ with respect to $\bar{\Phi}$. It follows that the Levi-Civita connection form $\con$ of $g$ satisfies
\[
\con=I\eta+J\nu-\chi=-(J\alpha+I\beta)-\chi.
\]
Recall that $\pi : S\Orb \to \Orb$ denotes the basepoint projection. Comparing with \cref{lemma:posweyltofinsler} we want to argue that there exists a unique $1$-form $\theta$ on $\Orb$ so that $\pi^*(\star_g\theta)=-(J\alpha+I\beta)$. Since $J\alpha+I\beta$ is semibasic for the projection $\pi$, it is sufficient to show that $J\alpha+I\beta$ is invariant under the $\mathrm{SO}(2)$ right action generated by the vector field $\dualcon$, where $(A,B,\dualcon)$ denote the vector fields dual to $(\alpha,\beta,\con)$. Denoting by $(X,H,V)$ the vector fields dual to $(\chi,\eta,\nu)$ on $S\Orb$, the identities
\[
\begin{pmatrix}
\alpha\\
\beta\\
\con\\ 
\end{pmatrix}=\begin{pmatrix} -\nu \\ -\eta \\ I\eta+J\nu-\chi\end{pmatrix}
\]
imply $\dualcon=-X$. Now observe the Bianchi identity 
\[
0=\d^2\eta=\left(J\chi-\d I\right)\wedge\eta\wedge\nu
\]
so that $XI=J$. Likewise we obtain
\[
0=\d^2\nu=-\left(\d J+I\chi\right)\wedge\eta\wedge\nu
\]
so that $XJ=-I$. From this we compute
\[
\mathrm{L}_X\left(I\eta+J\nu\right)=J\eta+I\nu-I\nu-J\eta=0,
\]
so that $-(J\alpha+I\beta)=\pi^*(\star_g\theta)$ for some unique $1$-form $\theta$ on $\Orb$ as desired. We obtain a Weyl structure defined by the pair $(g,\theta)$. Since
\begin{align*}
\d(\pi^*(\star_g\theta)-\con)&=\d\left(-(J\alpha+I\beta)+(J\alpha+I\beta+\chi)\right)=\d\chi=\nu\wedge\eta\\
&=\pi^*\left((K_g-\delta_g\theta)\area\right)=\left(K_g-\delta_g\theta\right)\circ \pi\,\nu\wedge\eta,
\end{align*}
we see that $K_g-\delta_g\theta=1$. Therefore $(g,\theta)$ is the natural gauge for the positive Weyl structure $[(g,\theta)]$. Finally, by construction, the Weyl structure $[(g,\theta)]$ is Besse. 
\newline
\newline
Conversely, let $\Orb=S^2(a_1,a_2)$ be a spindle orbifold as in $(2)$ with a positive Besse--Weyl structure $[(g,\theta)]$. Let $(g,\theta)$ be the natural gauge of $[(g,\theta)]$ and let $\pi : S\Orb \to \Orb$ denote the unit tangent bundle with respect to $g$. By \cite[Lemma~3.1]{La16} the unit-tangent bundle $S\Orb$ is a lens space of type $L(a_1+a_2,1)$. The canonical coframe on $S\Orb$ as explained in  \cref{riem} will be denoted by $(\alpha,\beta,\con)$. By \cref{lemma:posweyltofinsler} the $1$-forms $\chi,\eta,\nu$ on $S\Orb$ given by
\[
\chi:=\pi^*(\star\theta)-\con, \qquad \eta:=-\beta, \qquad \nu:=-\alpha
\]
define a generalized Finsler structure on $S\Orb$ of constant Finsler--Gauss cur\-va\-ture $K=1$, i.e. they satisfy the structure equations 
\begin{equation}\label{strucfin}
\d\chi=-\eta\wedge\nu, \qquad \d\eta=-\nu\wedge(\chi-I\eta), \qquad \d\nu=-(\chi-J\nu)\wedge\eta,\\
\end{equation}
for some smooth functions $I,J :S\Orb \to \R$. Moreover they parallelise $S\Orb$ and have the property that the leaves of the foliation $\mathcal{F}_g:=\left\{\chi,\eta\right\}^{\perp}$ are tangential lifts of maximal oriented geodesics of the Weyl connection ${}^{(g,\theta)}\nabla$ on $\Orb$. Since this connection is Besse by assumption, all of these leaves are circles. It follows from a theorem by Epstein \cite{MR0288785} that the leaves are the orbits of a smooth, almost free $S^1$-action. Since $a_1+a_2$ is odd, $S\Orb$ admits a normal covering by a space $M\cong L(2,1)\cong \mathbb{RP}^3$ with deck transformation group $\bar \Gamma$ isomorphic to $\Z_{(a_1+a_2)/2}$. The lifts of $\chi,\eta,\nu$ to $M$, which we denote by the same symbols, define a generalized Finsler structure on $M$ of constant Finsler--Gauss curvature 1. Moreover, the $S^1$-action on $S\Orb$ lifts to a smooth, almost free $S^1$-action on $M$ whose orbits are again the leaves of the foliation $\left\{\chi,\eta\right\}^{\perp}$. The leaves of the foliation $\mathcal{F}_t:=\left\{\eta,\nu\right\}^{\perp}$ correspond to (the lifts of) the fibres of the projection $S\Orb \To \Orb$ (to $M$) and are in particular also all circles. We can cover the space $M$ further by $S^3$ and lift the $S^1$-action and the foliations $\mathcal{F}_g$ and $\mathcal{F}_t$ to $S^3$. By the classification of Seifert fibrations of lens spaces quotienting out the foliations $\mathcal{F}_t$ and $\mathcal{F}_g$ of $S\Orb$, $M$ and $S^3$ yields a diagram of maps as follows (cf. \cref{sub:background_orbifolds} and e.g. \cite{GL16})
\[
	\begin{xy}
		\xymatrix
		{
			\tilde{\Orb} \cong S^2(a_1/\teiler,a_2/\teiler) \ar[d] & \tilde M \cong S^3 \ar[l] \ar[r] \ar[d] & \tilde{\Orb}_g \cong S^2(k_1,k_2) \ar[d]  \\
		  \bar{\Orb} \cong S^2(a_1/a,a_2/a) \ar[d] & M\cong L(2,1) \ar[l] \ar[r]^{\tau} \ar[d] & \bar{\Orb}_g \cong S^2(k'k_1,k'k_2) \ar[d]  \\
		 \Orb\cong S^2(a_1,a_2)  &S\Orb\cong L(a_1+a_2,1) \ar[l]  \ar[r]  & \Orb_g \cong S^2(kk_1,kk_2) 
		}
	\end{xy}
\]
 with $a| \gcd(a_1,a_2)=\teiler\in \{1,2\}$, $\gcd(k_1,k_2)=1$, $k'|2$ and $k | |\Gamma|= (a_1+a_2)/2$. Here the horizontal maps are smooth orbifold submersions, and the vertical maps are coverings (of manifolds in the middle and of orbifolds on the left and the right). Moreover, the deck transformation groups in the middle descend to deck transformation groups of the orbifold coverings. We claim that $a=1$. To prove this we can assume that $\teiler=2$. In this case the co-prime numbers $a_1/\teiler$ and $a_2/\teiler$ have different parity by our assumption that $\teiler^3 | a_1a_2$. Since $a_1/a+a_2/a$ has to be even by \cref{lem:seifert_on_rp3}, it follows that $a=1$ as claimed.

The involution
\[
	\begin{array}{cccl}
		  i: & S\Orb				& \rightarrow 	& S\Orb \\
		         & (x,v) 	&  \mapsto & (x,-v).
	\end{array}
\]
maps fibres of $\mathcal{F}_g$ and $\mathcal{F}_t$ to fibres of $\mathcal{F}_g$ and $\mathcal{F}_t$, respectively, and descends to a smooth orbifold involution $i$ of $\Orb_g$. We claim that the same argument as in \cite{La16} shows that $i$ does not fix the singular points on $\Orb_g$. Here we only sketch the ideas and refer to \cite{La16} for the details: If $a_1$ and $a_2$ are odd then $i$ acts freely on $\Orb_g$, and in this case nothing more has to be said. On the other hand, if $a_1$ and $a_2$ are even, then any geodesic that runs into a singular point is fixed by the action of $i$ on $\Orb_g$. In this case one first has to show that the lift $\tilde i :\tilde M \To \tilde M$ of $i$ to the universal covering of $S\Orb$ commutes with the deck transformation group $\Gamma$ of the covering $\tilde M \To S\Orb$. This can be shown based on the observation that a fibre of $\mathcal{F}_t$ on $S^3$ over the singular point of $\Orb$, together with its orientation, is preserved by both $\Gamma$ and $\tilde i$ (see \cite[Lemma~3.4]{La16} for the details). Now, if $a_1$ and $a_2$ are both even and a singular point on $\Orb_g$ is fixed by $i$, then there also exists a fibre of $\mathcal{F}_g$ on $S^3$ which is invariant under both $\Gamma$ and $\tilde i$. However, in this case only $\Gamma$ preserves the orientation of this fibre, whereas $\tilde i$ reverses its orientation. This leads to a contradiction to the facts that $|\Gamma|=a_1+a_2>2$ and that $\Gamma$ commutes with $\tilde i$ (see \cite[Lemma~3.5]{La16} for the details).

Since $i$ preserves the orbifold structure of $\Gamma_g$ and does not fix its singular points, it has to interchange the singular points. In particular, this implies that $kk_1=kk_2$, and hence $k_1=k_2=1$. Therefore the foliation $\mathcal{F}$ on $S^3$ is the Hopf-fibration and we must have $k'=1$ by \cref{lem:seifert_on_rp3}. In other words, $\bar\Orb_g$ is a smooth $2$-sphere without singular points and $\tau : M \to \bar\Orb_g=S^2$ is a smooth submersion. Consider the map
\[
\aligned
\Phi : M &\to TS^2\\
u &\mapsto -\tau^{\prime}_u(X(u)).\\
\endaligned
\] 
Then by~\cite[Proposition 1]{Bry} $\Phi$ immerses each $\tau$-fibre $\tau^{-1}(x)$ as a curve in $T_xS^2$ that is strictly convex towards $0_x$. The number of times $\Phi(\tau^{-1}(x))$ winds around $0_x$ does not depend on $x$. Since both $M$ and $\mathbb{S}TS^2$ are diffeomorphic to $L(2,1)$, the same argument as above proves that $\Phi$ is one-to-one, and so this number is one. Therefore, by ~\cite[Proposition 2]{Bry} $\Phi(M)$ is a Finsler structure on $S^2$. Moreover, $S^2$ can be oriented in such a way that the $\Phi$-pullback of the canonical coframing induced on $\Phi(M)$ agrees with $(\chi,\eta,\nu)$. In particular, this implies that the Finsler structure satisfies $K=1$ and has periodic geodesic flow. Moreover, because of $a=1$ we have $\bar \Orb = \Orb$ and therefore the preimages of the leaves of $\mathcal{F}_t$ under the covering $M \To S\Orb$ are connected. Since the covering  $M \To S\Orb$ is $(a_1+a_2)/2$-fold, so is its restriction to the fibres of $\mathcal{F}_t$. Therefore, $p:=(a_1+a_2)/2$ is the minimal number for which the geodesic flow of the Finsler structure on $S^2$ is $2\pi p$-periodic. The structure of $\bar \Orb$ implies that all closed geodesics of $S^2$ have length $2\pi p$ except at most two exceptions, which are $q:=a_1$ and $2q-p=a_2$ times shorter than the regular geodesics. In particular, the shortest geodesic has length $2\pi p/q=2\pi \frac{a_1+a_2}{2a_1}$ as claimed.

Finally, going through the proof shows that an isometry between two Finsler metrics as in the statement of Theorem induces a diffeomorphism between the corresponding spindle orbifolds that pulls back the two natural gauges onto each other, and vice versa. Hence, since such a pullback of a natural gauge is a natural gauge, the last statement of the Theorem follows from uniqueness of the natural gauge of a given Besse--Weyl structure.
\end{proof}

\section{Construction of Examples} \label{sec:examples}

In this section we exhibit our duality result to construct a $2$-dimensional family of deformations of a given rotationally symmetric Finsler metric on $S^2$ of constant curvature and all geodesics closed through metrics with the same properties. On the Besse-Weyl side these deformations correspond to deformations through Besse-Weyl structures in a fixed projective equivalence class. 

\subsection{The twistor space}\label{subsec:twistor}

Inspired by the twistorial construction of ho\-lo\-mor\-phic projective structures by Hitchin~\cite{MR699802} and LeBrun~\cite{thesislebrun}, it was shown in~\cite{MR728412,MR812312} how to construct a ``twistor space'' for smooth projective structures. Here we restrict our description to the case of an oriented surface $M$. Let $\mathsf{J}_+(M) \to M$ denote the fibre bundle whose fibre at $x \in M$ consists of the orientation compatible linear complex structures on $T_xM$. By definition, the sections of $\mathsf{J}_+(M) \to M$ are in bijective correspondence with the (almost) complex structures on $M$ that induce the given orientation. 

The choice of a torsion-free connection $\nabla$ on $TM$ allows to define an integrable almost complex structure on $\mathsf{J}_+(M)$ which depends only on the~\textit{projective equivalence} class $[\nabla]$ of $\nabla$. The projective equivalence class $[\nabla]$ of $\nabla$ consists of all torsion-free connections on $TM$ having the same unparametrised geodesics as $\nabla$. We refer to the resulting complex surface $\mathsf{J}_+(M)$ as the~\textit{twistor space} of $(M,[\nabla])$. In~\cite{MR3144212}, it is shown that a Weyl connection in the projective equivalence class $[\nabla]$ corresponds to a section of $\mathsf{J}_+(M) \to M$ whose image is a holomorphic curve. In the case of the $2$-sphere $S^2$ equipped with the projective structure arising from the Levi-Civita connection of the standard metric -- or equivalently, $\mathbb{CP}^1$ equipped with the projective structure arising from the Levi-Civita connection of the Fubini--Study metric -- the twistor space $\mathsf{J}_+(S^2)$ is biholomorphic to $\mathbb{CP}^2\setminus\mathbb{RP}^2$. Here we think of $\mathbb{RP}^2$ as sitting inside $\mathbb{CP}^2$ via its standard real linear embedding. As a consequence, one can show that the Weyl connections on $S^2$ whose geodesics are the great circles are in one-to-one correspondence with the smooth quadrics in $\mathbb{CP}^2\setminus\mathbb{RP}^2$, see~\cite{MR3144212}. Using our duality result, this recovers on the Finsler side Bryant's classification of Finsler structures on $S^2$ of constant curvature $K=1$ and with linear geodesics~\cite{Bry}. 

The construction of the twistor space can still be carried out for the case of a projective structure $[\nabla]$ on an oriented orbifold $\mathcal{O}$. Again, sections of $\mathsf{J}_+(\mathcal{O})\to \mathcal{O}$ having holomorphic image correspond to Weyl connections in $[\nabla]$. Since the spindle orbifold $S^2(a_1,a_2)$ may also be thought of as the~\textit{weighted projective line} $\mathbb{CP}(a_1,a_2)$ with weights $(a_1,a_2)$ (see \cref{sub:fubiny-study-metric}), one would expect that $\mathsf{J}_+(\mathbb{CP}(a_1,a_2))$ can be embedded holomorphically into the weighted projective plane, where we equip $\mathbb{CP}(a_1,a_2)$ with the projective structure arising from the Levi-Civita connection of the Fubini--Study metric. This is indeed the case as we will show in \cref{subsec:twistorembedding}. However, a difficulty that arises is that there is more than one natural candidate for the Fubini--Study metric on $\mathbb{CP}(a_1,a_2)$. We will next identify the correct metric for our purposes. 

\subsection{The Fubini--Study metric on the weighted projective line}\label{sub:fubiny-study-metric}

The \textit{(complex) weighted projective space} is the quotient of $\C^n\setminus\{0\}$ by $\C^*$, where $\C^*$ acts with weights $(a_1,\ldots,a_n) \in \mathbb{N}^n$, that is, by the rule
\[
z \cdot (z_1,\ldots,z_n)=(z^{a_1}z_1,\ldots,z^{a_n}z_n)
\]
for all $z \in \C^*$ and $(z_1,\ldots,z_n) \in \mathbb{C}^n\setminus\{0\}$. It inherits a natural quotient complex structure from $\Co^n$. We denote the projective space with weights $(a_1,\ldots,a_n)$ by $\mathbb{CP}(a_1,\ldots,a_n)$. Clearly, taking all weights equal to one gives ordinary projective space and for $n=2$ with weights $(a_1,a_2)$ we obtain the spindle orbifold $S^2(a_1,a_2)$. To omit case differentations we will henceforth restrict to the case where the pair $(a_1,a_2)$ is co-prime with $a_1\geqslant a_2$ and both numbers odd.

For what follows we would like to have an explicit Besse orbifold metric on $\mathbb{CP}(a_1,a_2)$ which induces the quotient complex structure of $\mathbb{CP}(a_1,a_2)$. The quotient Besse orbifold metric on $S^2(a_1,a_2)$ described in \cref{sub:background_orbifolds} satisfies this condition if and only if $a_1=a_2$. Abstractly, the existence of such a metric follows from the uniformisation theorem for orbifolds \cite{MR1631588} (see also \cite[Theorem~7.8.]{MR3692897}). In fact, since the biholomorphism group of $\mathbb{CP}(a_1,a_2)$ for $(a_1,a_2)\neq (1,1)$ contains a unique subgroup isomorphic to $S^1$, it even follows that such a metric can be chosen to be rotationally symmetric. We are now going to describe an orbifold metric with these properties which, in addition, will have strictly positive Gauss curvature. For this purpose it is convenient to describe the weighted projective line $\mathbb{CP}(a_1,a_2)$ as a quotient of $\SU$. In particular, we will identify $\SU$ as an $(a_1+a_2)$-fold cover of the unit tangent bundle of $\mathbb{CP}(a_1,a_2)$.

We consider
\[
\mathrm{SU}(2):=\left\{\begin{pmatrix} \varz & -\ov\varw \\ \varw & \ov \varz \end{pmatrix} : (\varz,\varw) \in \C^2, |\varz|^2+|\varw|^2=1\right\}
\]
and think of $\U$ as the subgroup consisting of matrices of the form
\[
\e^{\i\vartheta}\simeq \begin{pmatrix} \e^{-\i\vartheta} & 0 \\ 0 & \e^{\i\vartheta}\end{pmatrix} 
\]
for $\vartheta \in \R$. Consider the smooth $S^1$-action 
\begin{equation}\label{eq:defs1action}
T_{\e^{\i\vartheta}}:=L_{\e^{\i(a_1-a_2)\vartheta/2}}\circ R_{\e^{\i(a_1+a_2)\vartheta/2}} : \SU \to \SU 
\end{equation}
for $\vartheta \in \R$ and where $L_g$ and $R_g$ denote left -- and right multiplication by the group element $g \in \SU$. Explicitly, we have
\[
T_{\e^{\i\vartheta}}\left(\esu\right)=\begin{pmatrix} \e^{-\i a_1\vartheta}\varz & -\e^{\i a_2 \vartheta} \ov\varw \\\e^{-\i a_2\vartheta} \varw & \e^{\i a_1\vartheta}\ov \varz\end{pmatrix},
\]
and hence the corresponding quotient can be identified with the weighted projective line $\mathbb{CP}(a_1,a_2)$. 

Recall that the Maurer--Cartan form $\varrho$ is defined as $\varrho_g(v):=\left(L_{g^{-1}}\right)^{\prime}_g(v)$ where $v \in T_g\SU$. Writing the Maurer--Cartan form $\varrho$ of $\SU$ as
\[
\varrho=\begin{pmatrix} -\i\kappa & -\ov\varphi \\ \varphi & \i\kappa\end{pmatrix}
\]
for a real-valued $1$-form $\kappa$ and a complex-valued $1$-form $\varphi$ on $\SU$, the structure equation $\d\varrho+\varrho\wedge\varrho=0$ is equivalent to
\[
\d\varphi=-2\i\kappa\wedge\varphi\quad \text{and} \quad \d\kappa=-\i\varphi\wedge\ov{\varphi}
\]
Since
\[
\varrho=\esu^{-1}\d \esu
\]
we also obtain
\[
\d z=-\ov{w}\varphi-\i z \kappa\quad \text{and}\quad \d w= \ov{z}\varphi-\i w \kappa
\]
as well as
\begin{equation}\label{eq:formsdzdw}
\varphi=z\d w-w \d z\quad \text{and} \quad \kappa=\i\left(\ov{z}\d z+\ov{w}\d w\right).
\end{equation}

In order to compute a basis for the $1$-forms that are semi-basic for the projection $\pi_{a_1,a_2} : \SU \to \mathbb{CP}(a_1,a_2)$, we evaluate the Maurer--Cartan form on the infinitesimal generator $\dualcon:=\left.\frac{\d}{\d t}\right|_{t=0}T_{\e^{\i t}}$ of the $S^1$-action. We obtain
\begin{align*}
\varrho(\dualcon)&=\esu^{-1}\left.\frac{\d}{\d t}\right|_{t=0}\begin{pmatrix} \e^{-\i a_1 t}\varz & -\e^{\i a_2 t} \ov\varw \\\e^{-\i a_2 t} \varw & \e^{\i a_1 t}\ov \varz\end{pmatrix}\\
&=\begin{pmatrix} -\i(a_1|\varz|^2+a_2|\varw|^2) & \i (a_1-a_2)\ov{\varz\varw} \\ \i(a_1-a_2) \varz \varw & \i(a_1|\varz|^2+a_2|\varw|^2) \end{pmatrix},
\end{align*}
so that $\kappa(\dualcon)=\mathscr{U}$ and $\varphi(\dualcon)=\mathscr{V}$, where
\[
\mathscr{U}=a_1|z|^2+a_2|w|^2\quad\text{and} \quad \mathscr{V}=\i(a_1-a_2)zw. 
\]
Consequently, we see that the complex-valued $1$-form 
\begin{equation}\label{eq:defomega}
\omega=\mathscr{U}\varphi-\mathscr{V}\kappa
\end{equation}
satisfies $\omega(\dualcon)=0$ and hence -- by definition -- is semibasic for the projection $\pi_{a_1,a_2} : \SU \to \mathbb{CP}(a_1,a_2)$. Because of the left-invariance of $\varrho$ we have $T_{\e^{\i\vartheta}}^*\varrho=(R_{\e^{\i(a_1+a_2)\vartheta/2}})^*\varrho$ and hence
\begin{align*}
(T_{\e^{\i\vartheta}})^*\varrho&=
\begin{pmatrix} -\i\kappa & \e^{\i(a_1+a_2)\vartheta}\ov\varphi \\ \e^{-\i(a_1+a_2)\vartheta}\varphi & \i\kappa \end{pmatrix},
\end{align*}
where we have used the equivariance property $R_g^*\varrho=g^{-1}\varrho g$ which holds for all $g \in \SU$. Since $(T_{\e^{\i\vartheta}})^*\mathscr{U}=\mathscr{U}$ and $(T_{\e^{\i\vartheta}})^*\mathscr{V}=\e^{-\i(a_1+a_2)\vartheta}\mathscr{V}$, we obtain
\begin{align}
\label{eq:equiomega}(T_{\e^{\i\vartheta}})^*\omega&=\e^{-\i(a_1+a_2)\vartheta}\omega\\
\label{eq:equipsi}(T_{\e^{\i\vartheta}})^*\con&=\con,
\end{align}
where 
\begin{equation}\label{eq:defpsi}
\con=\kappa/\mathscr{U}.
\end{equation} Infinitesimally we obtain
\[
\d\omega=-\i(a_1+a_2)\con\wedge\omega\quad \text{and}\quad \d\con=-\frac{\i K_g}{2(a_1+a_2)}\omega\wedge\ov{\omega}
\]
with $K_g=2(a_1+a_2)/\mathscr{U}^3$. For later usage we also record the identities
\begin{align}
\label{eq:iddz}\d z&=-(\ov{w}/\mathscr{U})\omega-\i a_1 z\con,\\
\label{eq:iddw}\d w&=(\ov{z}/\mathscr{U})\omega-\i a_2 w\con. 
\end{align}
Observe that if we write $\omega=\alpha+\i\beta$ for real-valued $1$-forms $\alpha,\beta$ on $\SU$, then we obtain the structure equations
\begin{align*}
\d\alpha=-(a_1+a_2)\beta\wedge\con, \quad \d\beta&=-(a_1+a_2)\con\wedge\alpha, \quad \d\con=-\frac{K_g}{(a_1+a_2)}\alpha\wedge\beta. 
\end{align*}
Now since $(a_1,a_2)$ are co-prime, the cyclic group $\Z_{a_1+a_2}\subset S^1$ of order $ a_1+a_2$ acts freely on $\SU$. Therefore, the quotient $\SU/\Z_{a_1+a_2}$ is a smooth manifold equipped with a smooth action of $S^1/\Z_{a_1+a_2}\simeq S^1$ which we denote by $\un{T}_{\e^{\i\vartheta}}$. Writing $\upsilon : \SU \to \SU/\Z_{a_1+a_2}$ for the quotient projection, we have the equivariance property 
\[
\upsilon \circ T_{\e^{\i\vartheta}}=\un{T}_{\e^{\i(a_1+a_2)\vartheta}}\circ \upsilon
\]
for all $\e^{\i\vartheta} \in S^1$. Denoting the infinitesimal generator of the $S^1$ action $\un{T}_{\e^{\i\vartheta}}$ by $\un{\dualcon}$ we thus obtain
\[
\upsilon^{\prime}(\dualcon)=(a_1+a_2)\un{\dualcon}.
\]
Likewise, denoting the framing of $\SU$ that is dual to $(\alpha,\beta,\con)$ by $(A,B,\dualcon)$, the equivariance properties~\eqref{eq:equiomega},\eqref{eq:equipsi} imply that we obtain unique well defined vector fields $\un{A},\un{B}$ on $\SU/\Z_{a_1+a_2}$ so that 
\[
\upsilon^{\prime}(A)=\un{A}\quad\text{and}\quad \upsilon^{\prime}(B)=\un{B}.
\]
In particular, the structure equations for $(\alpha,\beta,\con)$ imply the commutator relation
\[[\un{\dualcon},\un{A}]=\un{B}\quad\text{and}\quad [\un{\dualcon},\un{B}]=-\un{A} \quad \text{and} \quad[\un{A},\un{B}]=K_g\un{\dualcon},
\]
where, by abuse of notation, we here think of $K_g$ as a function on the quotient $\SU/\Z_{a_1+a_2}$. These commutator relations in turn imply that the coframing $(\un{\alpha},\un{\beta},\un{\con})$ of $\SU/\Z_{a_1+a_2}$ that is dual to $(\un{A},\un{B},\un{\dualcon})$ defines a generalized Finsler structure of Riemannian type. 

Exactly as in the proof of~\cref{thm:duality} it follows that there exists a unique orientation and orbifold metric $g$ on $\mathbb{CP}(a_1,a_2)$, so that $\pi^*g=\un{\alpha}\otimes\un{\alpha}+\un{\beta}\otimes\un{\beta}$ and so that the area form of $g$ satisfies $\pi^*\area=\un{\alpha}\wedge\un{\beta}$. Here $\pi : \SU/\Z_{a_1+a_2} \to \mathbb{CP}(a_1,a_2)$ denotes the quotient projection with respect to the $S^1$ action $\un{T}_{\e^{\i\vartheta}}$. Moreover, the map
\[
\Phi : \SU/\Z_{a_1+a_2} \to T\mathbb{CP}(a_1,a_2), \quad u\mapsto \pi^{\prime}_u(\un{A}(u))
\]
is a diffeomorphism onto the unit tangent bundle $S\mathbb{CP}(a_1,a_2)$ of $g$ which has the property that the pullback of the canonical coframing on $S\mathbb{CP}(a_1,a_2)$ yields $(\un{\alpha},\un{\beta},\un{\con})$. Thus, we will henceforth identify the unit tangent bundle $S\mathbb{CP}(a_1,a_2)$ of $(\mathbb{CP}(a_1,a_2),g)$ with $\SU/\Z_{a_1+a_2}$.

We will next show that $g$ is a Besse orbifold metric. For an element $y$ in the Lie algebra $\su$ of $\SU$ we let $Y_y$ denote the vector field on $\SU$ generated by the flow $R_{\exp(ty)}$. Recall that the Maurer--Cartan form $\varrho$ satisfies 
$\varrho(Y_y)=y$ for all $y \in \su$. It follows that the basis
\[
e_1=\begin{pmatrix}0 & -1 \\ 1 & 0\end{pmatrix}\quad \text{and}\quad e_2=\begin{pmatrix}0 & \i \\ \i & 0\end{pmatrix}\quad \text{and} \quad e_3=\begin{pmatrix}-\i & 0 \\ 0 & \i\end{pmatrix}
\]
of $\su$ yields a framing $(Y_{e_1},Y_{e_2},Y_{e_3})$ of $\SU$ which is dual to the coframing $(\Re(\varphi),\Im(\varphi),\kappa)$. Therefore, using~\eqref{eq:defomega},~\eqref{eq:defpsi} and the definition of $\alpha,\beta$, we obtain $A=Y_{e_1}/\mathscr{U}$. The flow of $Y_{e_1}$ is given by $R_{\exp(te_1)}$ and hence periodic with period $2\pi$. Recall that the geodesic flow of the metric $g$ on $\mathbb{CP}(a_1,a_2)$ is $\un{A}$. Thinking of $\mathscr{U}$ as a function on $S\mathbb{CP}(a_1,a_2)$, we have 
\[
\un{A}=\upsilon^{\prime}(A)=\upsilon^{\prime}(Y_{e_1})/\mathscr{U}
\]
and hence $g$ is a Besse orbifold metric. 

In complex notation, we have $(\pi_{a_1,a_2})^*g=\omega\circ\ov{\omega}$, where $\circ$ denotes the symmetric tensor product and $\omega=\alpha+\i\beta$. The complex structure on $\mathbb{CP}(a_1,a_2)$ defined by $g$ and the orientation is thus characterized by the property that its $(1,\! 0)$-forms pull-back to $\SU$ to become complex multiples of $\omega$. In particular, this complex structure coincides with the natural quotient complex structure of $\mathbb{CP}(a_1,a_2)$ since $\omega$ is a linear combination of $\d z$ and $\d w$, see~\eqref{eq:formsdzdw}.

Finally, observe that $K_g$ is strictly positive. We have thus shown:
\begin{lemma}\label{lem:fubinistudymetric}
There exists a Besse orbifold metric $g$ and orientation on $\mathbb{CP}(a_1,a_2)$ so that $(\pi_{a_1,a_2})^*g=\alpha\otimes\alpha+\beta\otimes\beta$ and so that $(\pi_{a_1,a_2})^*\area=\alpha\wedge\beta$. This metric and orientation induce the quotient complex structure of $\mathbb{CP}(a_1,a_2)$. Moreover, $g$ has strictly positive Gauss curvature $K_g=2(a_1+a_2)/(a_1|z|^2+a_2|w|^2)^3$.
\end{lemma}
\begin{remark}
The reader may easily verify that in the case $a_1=a_2=1$ we recover the usual Fubini--Study metric on $\mathbb{CP}^1$. For this reason we refer to $g$ as the Fubini--Study metric of $\mathbb{CP}(a_1,a_2)$.
\end{remark}
\begin{figure}
\centering
\begin{subfigure}{.5\textwidth}
  \centering
  \includegraphics[width=.75\linewidth]{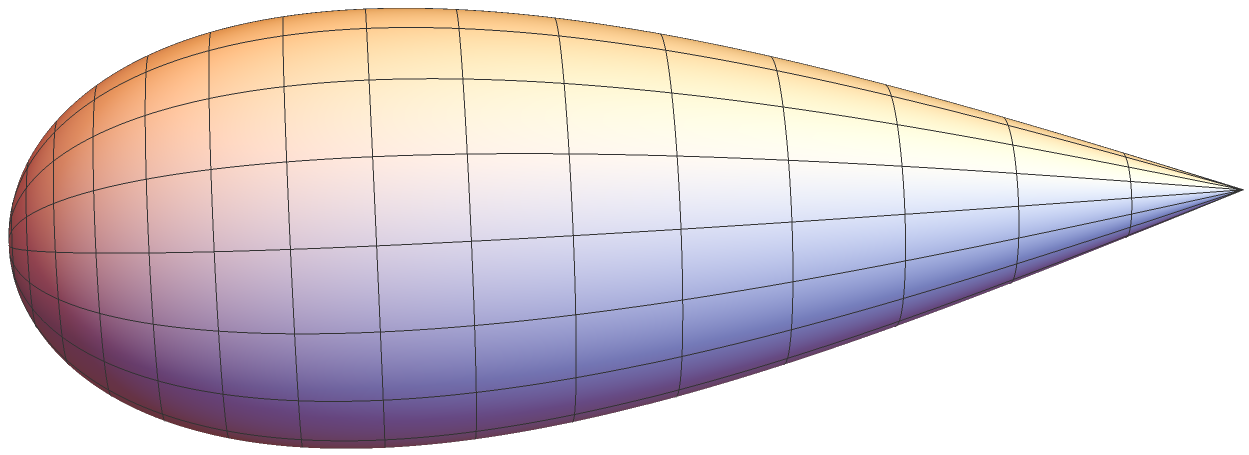}
\end{subfigure}%
\begin{subfigure}{.5\textwidth}
  \centering
  \includegraphics[width=.75\linewidth]{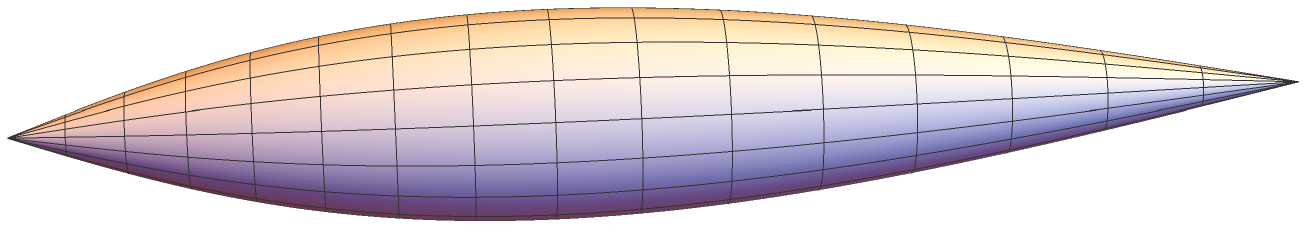}
\end{subfigure}
\caption*{Isometric embeddings of $\mathbb{CP}(3,1)$ and $\mathbb{CP}(5,3)$}
\end{figure}

\subsection{Constructing the twistor space}\label{sec:construcalmostcplx}

 We will next construct the twistor space $\mathsf{J}_+(\mathcal{O})$ in the case of the weighted projective line $\mathbb{CP}(a_1,a_2)$ and where $[\nabla]$ is the projective equivalence class of the Levi-Civita connection of the Fubini--Study metric on $\mathbb{CP}(a_1,a_2)$ constructed in \cref{lem:fubinistudymetric}. As we will see, in this special case the twistor space $\mathsf{J}_+(\mathbb{CP}(a_1,a_2))$ can indeed be embedded into the weighted projective plane $\mathbb{CP}(a_1,(a_1+a_2)/2,a_2)$. We will henceforth also write $\mathsf{J}_+$ for the twistor space, whenever the underlying orbifold is clear from the context. 

In the case of a Riemann surface $(M,J)$ the orientation inducing complex structures on $M$ are in one-to-one correspondence with the~\textit{Beltrami differentials} on $(M,J)$. A Beltrami differential $\mu$ on $M$ is a section of $B=K_{M}^{-1}\otimes \ov{K_M}$, where $K_M=(T^{*}_{\C}M)^{1,0}$ denotes the canonical bundle of $(M,J)$ with inverse $K_M^{-1}$ and where $\ov{K_M}$ denotes its complex conjugate bundle. The line bundle $B$ carries a natural Hermitian bundle metric $h$ and Beltrami differentials are precisely those sections which satisfy the condition $h(\mu,\mu)<1$ at each point of $M$. This identifies $\mathsf{J}_+(M)$ with the open unit disk bundle in $B$. We refer the reader to~\cite{MR1215481} for additional details. 

In the case of the orbifold $\mathbb{CP}(a_1,a_2)$, the orbifold canonical bundle with respect to the Riemann surface structure induced by the orientation and the Fubini--Study metric $g$ described in~\cref{lem:fubinistudymetric}, can be defined as a suitable quotient of $\SU\times \C$. 

Recall that the metric $g$ on $\mathbb{CP}(a_1,a_2)$ satisfies $(\pi_{a_1,a_2})^*g=\omega\circ\ov{\omega}$ and $(\pi_{a_1,a_2})^*\area=\frac{\i}{2}\omega\wedge\ov{\omega}$, where $\omega=\alpha+\i\beta$. In particular, the complex structure on $\mathbb{CP}(a_1,a_2)$ induced by $g$ and the orientation has the property that its $(1,\! 0)$-forms pull-back to $\SU$ to become complex multiples of $\omega$.

 Moreover, recall that from~\eqref{eq:equiomega} that we have $(T_{\e^{\i\vartheta}})^*\omega=\e^{-\i(a_1+a_2)\vartheta}\omega$. We thus define
\[
K_{\mathbb{CP}(a_1,a_2)}=\SU\times_{S^1}\C,
\]
where $S^1$ acts on $\SU$ by~\eqref{eq:defs1action} and on $\C$ with~\textit{spin} $(a_1+a_2)$, that is, by the rule
\[
\e^{\i\vartheta}\cdot z=\e^{\i(a_1+a_2)\vartheta}z. 
\]
Likewise, $K^{-1}_{\mathbb{CP}(a_1,a_2)}$ arises from acting with spin $-(a_1+a_2)$ and $\ov{K_{\mathbb{CP}(a_1,a_2)}}$ arises from the complex conjugate of the spin $(a_1+a_2)$ action, that is, also from the action with spin $-(a_1+a_2)$. Therefore, we obtain $B=\SU\times_{S^1} \C$, where now $S^1$ acts with spin $-2(a_1+a_2)$ on $\C$. The Hermitian bundle metric $h$ on $B$ arises from the usual Hermitian inner product on $\C$ and hence we obtain
\[
\mathsf{J}_+(\mathbb{CP}(a_1,a_2))=\SU\times_{S^1}\mathbb{D},
\]
where $\mathbb{D}\subset \C$ denotes the open unit disk and $S^1$ acts with spin $-2(a_1+a_2)$ on $\mathbb{D}$. 

We now define an almost complex structure $\mathfrak{J}$ on $\J$. On $\SU\times \mathbb{D}$ we consider the complex-valued $1$-forms
\[
\xi_1=\omega+\mu\ov{\omega}\quad\text{and}\quad \xi_2=\d\mu+2(a_1+a_2)\i\mu\con,
\]
where $\mu$ denotes the standard coordinate on $\mathbb{D}$. Abusing notation and writing $T_{\e^{\i\varphi}}$ for the combined $S^1$-action on $\SU\times \mathbb{D}$, we obtain
\[
(T_{\e^{\i\varphi}})^*\xi_1=\e^{-\i(a_1+a_2)\vartheta}\xi_1\quad \text{and}\quad \left(T_{\e^{\i\varphi}}\right)^*\xi_2=\e^{-2\i(a_1+a_2)\vartheta}\xi_2.
\]
Furthermore, by construction, the forms $\xi_1$ and $\xi_2$ are semi-basic for the projection $\SU\times \mathbb{D} \to \J$. It follows that there exists a unique almost complex structure $\mathfrak{J}$ on $\J$ whose $(1,\! 0)$-forms pull-back to $\SU\times \mathbb{D}$ to become linear combinations of $\xi_1$ and $\xi_2$. Finally, in~\cite[\S4.2]{MR4109900} it is shown that the so constructed almost complex structure agrees with the complex structure on the twistor space associated to $(\mathbb{CP}(p,q),[{}^g\nabla])$ where ${}^g\nabla$ denotes the Levi-Civita connection of $g$. 

\begin{remark}
More precisely, in~\cite[\S4.2]{MR4109900} only the case of smooth surfaces is considered, but the construction carries over to the orbifold setting without difficulty. 
\end{remark}

\subsection{Embedding the twistor space}\label{subsec:twistorembedding}

Recall that the twistor space of the $2$-sphere $\mathsf{J}_+(S^2)$ equipped with the complex structure coming from the projective structure of the standard metric maps biholomorphically onto $\mathbb{CP}^2\setminus \mathbb{RP}^2$. The map arises as follows. Consider $S^2$ as the unit sphere in $\R^3$ and identify the tangent space $T_eS^2$ to an element $e \in S^2$ with the orthogonal complement $\{e\}^{\perp}\subset \R^3$. Then an orientation compatible complex structure $J$ on $T_eS^2$ is mapped to the element $[v+iJv]\in \mathbb{CP}^2\setminus \mathbb{RP}^2$ where $v\in T_e S^2$ is any non-zero tangent vector. With respect to our present model of $\mathsf{J}_+(S^2)$ as an associated bundle this map takes the following explicit form
\[
		 \begin{array}{cccl}
		 \bihol : &\mathsf{J}_+(S^2)=\mathrm{SU}(2)\times_{S^1} \mathbb{D} 				& \To  	&	\Co\mathbb{P}^2 \\
		           		& [z:w:\mu] 	& \MTo  & [z^2-\mu\ov{w}^2:zw+\ov{z}\ov{w}\mu:w^2-\mu\ov{z}^2]
		 \end{array}	
\]
after applying a linear coordinate change (see \cref{appendix1}). In this new coordinate system the real projective space $\mathbb{RP}^2$ sits inside $\mathbb{CP}^2$ as the image of the unit sphere in $\C\times \R$ under the map $j=\pi\circ\hat j: \C\times \R \To \mathbb{CP}^2$ where $j: \C\times \R \To \C^3$ is defined as $j(z,t)= (z,it,\ov{z})$ and where $\pi:\C^3\backslash \{0\} \To \mathbb{CP}^2$ is the quotient projection. Note that for $\mu=0$ the map $\bihol$ restricts to the \emph{Veronese embedding} of $\mathbb{CP}^1$ into $\mathbb{CP}^2$.

We observe that in the weighted case the very same map $\bihol$ also defines a smooth map of orbifolds. In fact, we are going to show the following statement in a sequence of lemmas.


\begin{proposition}\label{prp:weighted_biho} The map
\[
\bihol : \mathsf{J}_+(\Co\mathbb{P}(a_1,a_2))=\mathrm{SU}(2)\times_{S^1} \mathbb{D} \To   \Co\mathbb{P}(a_1,(a_1+a_2)/2,a_2)
\]
defined by
\[
[z:w:\mu] \mapsto [z^2-\mu\ov{w}^2:zw+\ov{z}\ov{w}\mu:w^2-\mu\ov{z}^2]
\]
is a biholomorphism onto $\Co\mathbb{P}(a_1,(a_1+a_2)/2,a_2) \backslash j(S^2)$ where $j=\pi \circ \hat j$ as above. Moreover, $j(S^2)$ is a real projective plane $\mathbb{RP}^2((a_1+a_2)/2)$ with a cyclic orbifold singularity of order $(a_1+a_2)/2$.
\end{proposition}

\begin{remark}
More precisely, by biholomorphism, we mean a diffeomorphism $\bihol$ which is holomorphic in the sense that it is $(\mathfrak{J},J_0)$-linear. By this we mean  that it satisfies $J_0\circ\bihol^{\prime}=\bihol^{\prime}\circ\mathfrak{J}$, where $\mathfrak{J}$ denotes the almost complex structure defined on $\mathsf{J}_+(\mathbb{CP}(a_1,a_2))$ in \cref{sec:construcalmostcplx} and $J_0$ the standard complex structure on the weighted projective space $\mathbb{CP}(a_1,(a_1+a_2)/2,a_2)$. 
\end{remark}

In the following we also describe $\mathrm{SU}(2)\times_{S^1} \mathbb{D}$ as a quotient of $\Co^2\backslash\{0\} \times \Co$ by the respective weighted $\Co^*$-action.  Here $\lambda \in \Co^*$ acts as $\lambda/|\lambda|$ on the second factor. Then the map $\bihol$ is covered by the $\Co^*$-equivariant map
\[
		 \begin{array}{cccl}
		 \hat \bihol : &\Co^2\backslash\{0\} \times \mathbb{D} 				& \To  	&	\Co^3 \backslash\{0\}\\
		           		& (z,w,\mu) 	& \MTo  & (z^2-\mu\ov{w}^2,zw+\ov{z}\ov{w}\mu,w^2-\mu\ov{z}^2).
		 \end{array}	
\]
Since we already know that the map $\bihol$ is an immersion in the unweighted case, and since the $\C^*$-actions on $\Co^2\backslash\{0\} \times \mathbb{D}$ and on $\Co^3 \backslash\{0\}$ do not have fixed points, it follows that the map $\hat \bihol$, and hence also the map $\bihol$ in the weighted case, is an immersion as well. Alternatively, the same conclusion can be drawn from an explicit computation which shows that the determinant of the Jacobian of the map $\hat \bihol$ is given by $\mathrm{det}(J(z,w,\mu))=4(1-|\mu|^2)(|z|^2+|w|^2)^4$.
%

There are different ways to continue the proof of \cref{prp:weighted_biho}. For instance, one can show that the map $\bihol$ extends to a smooth orbifold immersion of a certain compactification of $\mathsf{J}_+$ onto $\Co\mathbb{P}(a_1,(a_1+a_2)/2,a_2)$ so that the complement of $\mathsf{J}_+$ is mapped onto $j(S^2)$. Compactness and the fact that $\Co\mathbb{P}(a_1,(a_1+a_2)/2,a_2)$ is simply connected as an orbifold then imply that this map is a diffeomorphism. Since we do not need such a compactification otherwise at the moment, we instead prove the proposition by hand, which, in total, is less work. More precisely, we proceed by proving the following three lemmas.

\begin{lemma} The image of~$\bihol:\mathsf{J}_+ \To \Co\mathbb{P}(a_1,(a_1+a_2)/2,a_2)$ is disjoint from $j(S^2)$.
\end{lemma}
\begin{proof} Suppose we have $[z:w:\mu]\in \mathsf{J}_+$ with $\bihol([z:w:\mu]) \in j(S^2)$.  We can assume that $|z|^2+|w|^2=1$. Then there exists some $\lambda \in \Co^*$ such that $\lambda^{2a_1}(z^2-\ov{w}^2\mu)=\ov{\lambda}^{2a_2}(\ov{w}^2-z^2\ov{\mu})$, $\lambda^{a_1+a_2}(zw+\ov{z}\ov{w}\mu)\in i \R$ and
\[
	1=(\lambda^{2a_1}(z^2-\ov{w}^2\mu))(\lambda^{2a_2}(w^2-\ov{z}^2\mu))+|\lambda^{a_1+a_2}(z+\ov{z}\ov{w}\mu)|^2=-\lambda^{2(a_1+a_2)}\mu.
\]
If $zw\neq 0$ the last two conditions imply that $|\mu|=1$, a contradiction. Let us assume that $z=0$. Then $w\neq 0$ and so the first condition implies that $-\lambda^{2a_1}\mu=\ov{\lambda}^{2a_2}$. Together with $1=-\lambda^{2(a_1+a_2)}\mu$ this also implies $|\mu|=1$. The same conclusion follows analogously in the case $z\neq 0$. Hence, in any case we obtain a contradiction and so the lemma is proven.
\end{proof}

\begin{lemma} The map $\bihol:\mathsf{J}_+ \To \Co\mathbb{P}(a_1,(a_1+a_2)/2,a_2)$ is injective.
\end{lemma}
\begin{proof} Suppose we have $[z:w:\mu],[z':w':\mu'] \in \mathsf{J}_+$ with $\bihol([z:w:\mu])=\bihol([z':w':\mu'])$. Again we can assume that $|z|^2+|w|^2=1$ and $|z'|^2+|w'|^2=1$. There exists some $\lambda \in \Co^{*}$ such that 
\[
		(z^2-\mu\ov{w}^2,zw+\ov{z}\ov{w}\mu,w^2-\mu\ov{z}^2)= \lambda^2 (z'^2-\mu'\ov{w'}^2,z'w'+\ov{z'}\ov{w'}\mu',w'^2-\mu'\ov{z'}^2).
\]
Computing the expression $z_2^2-z_1z_3$ on both sides implies $\mu = \lambda^{2(a_1+a_2)} \mu'$.
We set $z''=\lambda^{a_1}z'$, $w''=\lambda^{a_2}w'$, $\mu''=\mu=\lambda^{2(a_1+a_2)} \mu'$ and obtain
\begin{equation}\label{eq:double_prime}
		(z^2-\mu\ov{w}^2,zw+\ov{z}\ov{w}\mu,w^2-\mu\ov{z}^2)= (z''^2-\mu\ov{w''}^2,z''w''+\ov{z''}\ov{w''}\mu,w''^2-\mu\ov{z''}^2).
\end{equation}
Computing the expressions $z_1+\mu\ov{z}_3$ and $z_3+\mu\ov{z}_1$ on both sides yields
\[
		z^2(1-|\mu|^2)=z''^2(1-|\mu|^2), \; \;  w^2(1-|\mu|^2)=w''^2(1-|\mu|^2).
\]
Because of $|\mu|<1$ it follows that $z^2=\lambda^{2a_1}z'^2$ and $w^2=\lambda^{2a_2}w'^2$, and hence $z=\varepsilon_z\lambda^{a_1}z'=\varepsilon_z z''$ and $w=\varepsilon_w\lambda^{a_2}w'=\varepsilon_w w''$ for some $\varepsilon_z,\varepsilon_w \in \{\pm 1\}$. Plugging this into the third component of equation~\eqref{eq:double_prime} we get
\[
	zw+\ov{z}\ov{w}\mu=\varepsilon_z\varepsilon_w (zw+\ov{z}\ov{w}\mu).
\]
If $z\neq 0\neq w$ then the expression on the left hand side is non-trivial because of $|\mu|<1$, and then $\varepsilon_z$ and $\varepsilon_w$ have the same sign. In this case it follows, perhaps after replacing $\lambda$ by $-\lambda$, that $z=\lambda^{a_1}z'$, $w=\lambda^{a_2}w'$ and hence $[z:w:\mu]=[z':w':\mu']$. Otherwise we can draw the same conclusion, again perhaps after replacing $\lambda$ by $-\lambda$.
\end{proof}

\begin{lemma} The map $\bihol:\mathsf{J}_+ \To \Co\mathbb{P}(a_1,(a_1+a_2)/2,a_2) - j(S^2)$ is onto.
\end{lemma}

\begin{proof} Let $(z_1,z_2,z_3)\in \Co^3\backslash\{0\}$ which does not project to $j(S^2)$. We set $\mu:=z_2^2-z_1z_3$. Replacing $(z_1,z_2,z_3)$ by $\lambda (z_1,z_2,z_3)$ for some $\lambda \in \Co^*$ changes $\mu$ to $\lambda^{a_1+a_2} u$. We first want to show that there exists some $\lambda \in \R_{>0}$ so that after replacing $(z_1,z_2,z_3)$ by $\lambda (z_1,z_2,z_3)$ we have $\mu\in (-1,0]$ and
\begin{equation}\label{eq:condition_1}
		|z_1+\ov{z}_3\mu|+|z_3+\ov{z}_1\mu| = 1- |\mu|^2.
\end{equation}
To prove this we consider the cases $\mu=0$ and $\mu\neq 0$ separately.

Let us first assume that $\mu=0$. In this case we need to find some $\lambda \in (0,\infty)$ such that 
\[
		\lambda^{a_1}|z_1|+\lambda^{a_2}|z_3| = 1.
\]
If $|z_1|+|z_3|>0$, this is possible by the intermediate value theorem. Otherwise we have $z_1=z_3=0$ and hence also $z_2=0$ (recall that $\mu=z_2^2-z_1z_3=0$), a contradiction. 

In the case $\mu\neq 0$ we can also assume that $\mu=-1$. So we need to find $\lambda \in (0,1)$ such that
\[
		\mathscr{F}(\lambda):=\lambda^{a_1}|z_1-\lambda^{2a_2}\ov{z}_3|+ \lambda^{a_2}|z_3-\lambda^{2a_1}\ov{z}_1| = 1- \lambda^{2(a_1+a_2)}=:\mathscr{G}(\lambda).
\]
By the intermediate value theorem this is possible if $\mathscr{F}(1)>0$. Otherwise we have $z_1=\ov{z_3}$ and $z_2^2=|z_1|^2-1$. In the case $|z_1|\leq 1$ this implies $z_2=it\in i \R$ and $|z_1|^2+t^2=|z_1|^2-z_2^2=1$ in contradiction to our assumption that $(z_1,z_2,z_3)$ does not project to $j(S^2)$. Therefore we can assume that $|z_1|>1$ and $z_1=\ov{z_3}$, in which case we have $\mathscr{F}(1)=\mathscr{G}(1)=0$. In order to find an appropriate $\lambda$ in this case it is sufficient to show that $\mathscr{F}'(1)<\mathscr{G}'(1)$. A computation shows that $\mathscr{G}'(1)=-2(p+q)$ and $\mathscr{F}'(1)=-|z_1|2(p+q)$. Hence in any case we can assume that $\mu\in [0,1)$ and that~\eqref{eq:condition_1} holds.

Now we can choose $z,w \in \Co$ such that $z^2(1-|\mu|^2)=z_1+\ov{z}_3\mu$ and $w^2(1-|\mu|^2)=z_3+\ov{z}_1\mu$. By construction we have $|z|^2+|w|^2=1$, and $z^2-\ov{w}^2\mu=z_1$ and $w^2-\ov{z}^2\mu=z_3$. Moreover, we see that
\[
		(zw+\ov{z}\ov{w} \mu)^2=(z^2-\ov{w}^2\mu)(w^2-\ov{z}^2\mu)+\mu=z_1z_3+\mu=z_2^2.
\]
Perhaps after using our freedom to change the sign of $z$ we obtain $(zw+\ov{z}\ov{w} \mu)=z_2$, and hence $\bihol([z:w:u])=[z_1:z_2:z_3]$ as desired.
\end{proof}

We have shown that $\bihol$ is a bijective immersion onto the complement of $j(S^2)$ in $\Co\mathbb{P}(a_1,(a_1+a_2)/2,a_2)$. It follows from the local structure of such maps that the inverse is smooth as well. Hence, the map $\bihol$ is a diffeomorphism onto the complement of $j(S^2)$.

In order to complete the proof of \cref{prp:weighted_biho} it remains to verify that the map $\bihol$ is holomorphic. The map $\bihol$ is holomorphic if and only if it pulls back $(1,0)$-forms to $(1,0)$-forms. By definition the $(1,0)$-forms of $\mathsf{J}_+$ pull back to linear combinations of $\xi_1$ and $\xi_2$ on $\SU \times \D$. On the other hand, the $(1,0)$-forms on $\Co\mathbb{P}(a_1,(a_1+a_2)/2,a_2)$ pull back to the $(1,0)$-forms on $\Co^3\backslash \{0\}$ which vanish on the infinitesimal generator of the defining $\C^*$-action on $\Co^3\backslash \{0\}$. The latter are linear combinations of the complex valued $1$-forms
\[
\Pi_1=a_2z_3\d z_1-a_1z_1\d z_3 \quad \text{and} \quad \Pi_2=\left(\frac{a_1+a_2}{2}\right)z_2\d z_1-a_1z_1\d z_2.
\]
Hence, we need to show that $\Psi=\hat{\bihol}|_{\SU\times \D} : \SU \times \D \to \C^{3}\setminus\{0\}$ satisfies
\begin{equation}\label{eq:conbiholom}
\xi_1\wedge \xi_2 \wedge \Psi^*\Pi_1 = 0 \quad \text{and} \quad \xi_1\wedge \xi_2 \wedge \Psi^*\Pi_2 = 0.
\end{equation}
Recall the identities~\eqref{eq:iddz} and~\eqref{eq:iddw}
\begin{align*}
\d z&=-(\ov{w}/\mathscr{U})\omega-\i a_1 z\con,\\
\d w&=(z/\mathscr{U})\omega-\i a_2 w\con. 
\end{align*}
Using these identities a tedious -- but straightforward -- calculation gives
\begin{multline*}
\Psi^*\Pi_1=-\frac{2}{\mathscr{U}}\left(a_1(z^2-\mu\ov{w}^2)\ov{z}w+a_2(w^2-\mu\ov{z}^2)z\ov{w}\right)\xi_1\\
+\left(a_1(z^2-\mu\ov{w}^2)\ov{z}^2-a_2(w^2-\mu\ov{z}^2)\ov{w}^2\right)\xi_2
\end{multline*}
and
\begin{multline*}
\Psi^*\Pi_2=-\frac{1}{\mathscr{U}}\left(a_1(z^3\ov{z}+\mu w\ov{w}^3)+a_2(zw+\mu\ov{zw})z\ov{w}\right)\xi_1\\
+\frac{1}{2}\left(a_1(\mu\ov{z}\ov{w}^2-|w|^2z-2|z|^2z)\ov{w}-a_2(\mu\ov{zw}+zw)\ov{w}^2\right)\xi_2,
\end{multline*}
thus~\eqref{eq:conbiholom} is satisfied and $\Xi$ is a biholomorphism. This finishes the proof of \cref{prp:weighted_biho}.

\subsection{Projective transformations} Let $\mathcal{O}$ be an orbifold equipped with a torsion-free connection $\nabla$ on its tangent bundle. A projective transformation for $(\mathcal{O},\nabla)$ is a diffeomorphism $\Psi : \mathcal{O} \to \mathcal{O}$ which sends geodesics of $\nabla$ to geodesics of $\nabla$ up to parametrisation. In the case where $\mathcal{O}$ is a smooth manifold the group of projective transformations of $\nabla$ is known to be a Lie group (see for instance~\cite{MR1336823}). In our setting, the projective transformations of the Besse orbifold metric on $\mathbb{CP}(a_1,a_2)$ also form a Lie group, since the automorphisms of the associated generalized path geometry form a Lie group, see~\cite{MR3586335} for details. Moreover, a vector field is called projective if its (local) flow consists of projective transformations. Clearly, if $\nabla$ is a Levi-Civita connection for some Riemannian metric $g$, then every Killing vector field for $g$ is a projective vector field. The set of vector fields for $\nabla$ form a Lie algebra given by the solutions of a linear second order PDE system of finite type. In the case of two dimensions and writing a projective vector field as $W=W^1(x,y)\frac{\partial}{\partial x}+W^2(x,y)\frac{\partial}{\partial y}$ for local coordinates $(x,y) : U \to \R^2$ and real-valued functions $W^i$ on $U$, the PDE system is~\cite{MR2368987}
\begin{align}\label{eq:pdeprojvec}
\begin{split}
0&=W^2_{xx}-2R^0W^1_x-R^1W^2_x+R^0W^2_y-R^0_xW^1-R^0_yW^2,\\
0&=-W^1_{xx}+ 2W^2_{xy}-R^1W^1_x-3R^0W^1_y-2R^2W^2_x-R^1_xW^1-R^1_yW^2,\\
0&=-2W^1_{xy}+W^2_{yy}-2R^1W^1_y-3R^3W^2_x-R^2W^2_y-R^2_xW^1-R^2_yW^2,\\
0&=-W^1_{yy}+R^3W^1_x-R^2W^1_y-2R^3W^2_y-R^3_xW^1-R^3_yW^2,
\end{split}
\end{align}
where 
\[
R^0=-\Gamma^2_{11},\quad R^1=\Gamma^1_{11}-2\Gamma^2_{12}, \quad R^2=2\Gamma^1_{12}-\Gamma^2_{22},\quad R^3=\Gamma^1_{22}
\]
and where $\Gamma^i_{jk}$ denote the Christoffel symbols of $\nabla$ with respect to $(x,y)$.

In order to show that the deformations we are going to construct in  \cref{sub:deformations} are nontrivial, we need to know that the identity component of the group of projective transformations of $(\mathbb{CP}(a_1,a_2),g)$ consists solely of isometries. Up to rescaling, any rotationally symmetric Besse metric on $\mathbb{CP}(a_1,a_2)$ is isometric to the metric completion of one of the following examples (see \cite[Section~2.2]{La16} and \cite[Thm.~4.13]{MR496885}): let $h: [-1,1] \To (- \frac{a_1+a_2}{2},\frac{a_1+a_2}{2})$ be a smooth, odd function with $h(1)=\frac{a_1-a_2}{2}=-h(-1)$ and let a Riemannian metric on $(0,\pi)\times ([0,2\pi]/0 \sim 2\pi) \ni (r,\phi)$ be defined by
\begin{equation}\label{eq:metric_revolution2}
			g_h=\left( \frac{a_1+a_2}{2}+h(\cos(r))\right)^2 d\theta^2 + \sin^2(r) d\phi^2.
\end{equation}
Our specific Besse orbifold metric $g$ on $\mathbb{CP}(a_1,a_2)$ takes the form $g_h/4$ with $h(x)=\frac{1}{2}(a_1-a_2)x$ with respect to the parametrization 
\begin{equation}\label{eq:spindlecoord}
[z:w]=\left[\cos(r/2)\e^{-\i\phi/(a_1+a_2)}:\sin(r/2)\e^{\i\phi/(a_1+a_2)}\right] 
\end{equation}
where $(r,\phi) \in (0,\pi)\times([0,2\pi]/0 \sim 2\pi)$.

\begin{lemma}\label{lem:groupprojtrans}
In our setting where $a_1>a_2$ are co-prime and odd the identity component of the group of smooth projective transformations of a rotationally symmetric Besse metric on $\mathbb{CP}(a_1,a_2)$ consists only of isometries. 
\end{lemma}
\begin{proof} Since $a_1>a_2\geqslant 1$ every projective transformation $\tau$ fixes the singular point of order $a_1$, the \emph{northpole} ($r=0$), and hence also its antipodal point of order $a_2$, the \emph{southpole} ($r=\pi$). Moreover it leaves the unique exceptional geodesic, the \emph{equator} ($r=\pi/2$), invariant. After composition with an isometry we can assume that $\tau$ fixes a point $x_0$ on the equator. Then $\tau$ also leaves invariant the minimizing geodesic between $x_0$ and the northpole. Because of $a_1>2$ it follows that the differential of $\tau$ at the northpole is a homotethy, i.e. it scales by some factor $\lambda > 0$. Therefore, $\tau$ in fact leaves invariant all geodesics starting at the northpole and consequently fixes the equator pointwise. In particular, the derivative of $\tau$ in the east-west direction along the equator is the identity. We write our Besse orbifold metric in coordinates as in (\cref{eq:metric_revolution2}). Let $x$ be some point on the equator. We can assume that is has coordinates $(r,\phi)=(\pi/2,0)$. We look at regular unit-speed geodesics $\gamma(s)=(r(s),\phi(s))$ with $\phi'(0)>0$ that start at $x$ and do not pass the singular points. Let $r_m$ be the maximal (or minimal) latitude attained by such a geodesic. By \cite[Theorem~4.11]{MR496885} this latitude is attained at a unique value of $s$ during one period. By symmetry and continuity the corresponding $\phi$-coordinate $\phi_m$ is constant as long as $r'(0)$ does not change its sign. According to Clairaut's relation we have $\sin^2(r)\phi'(s)=\sin(r_m)$ along $\gamma$ and the geodesic oscillates between the parallels $r=r_m$ and $r=\pi-r_m$ \cite[p.~101]{MR496885}. Let $\tilde \gamma (s) =(\tilde r(s),\tilde \phi(s))$ be the unit-speed parametrization of the geodesic $\tau(\gamma)$. Suppose the differential of $\tau$ at $x$ scales by a factor of $\lambda'>0$ in the north-south direction. Then we have
\[
   \tilde \phi'(0)= \frac{\phi'(0)}{\sqrt{\lambda'^2+(1-\lambda'^2) \phi'(0)^2}}
\]
and a corresponding relation between $\sin(\tilde r_m)$ and $\sin(r_m)$ by Clairaut's relation. Therefore, $\tau$ maps the curve $c:[0,\pi/2] \ni t \mapsto (t,\phi_m)$ to the curve 
\[
		\tilde c: [0,\pi/2] \ni t \mapsto \arcsin \left(  \frac{\sin(t)}{\sqrt{\lambda'^2+(1-\lambda'^2) \sin(t)^2}} ,\phi_m \right) 
\]
with $\tilde c'(0)=(1/\lambda',0)$. Hence, we have $\lambda=1/\lambda'$. In particular, in our coordinates the differential of $\tau$ looks the same at every point of the equator. It follows that $\tau=\tau_\lambda$ maps the $r$-parallels to the $\tilde r$-parallels, where
\[
		\sin(\tilde r) = \frac{\lambda \sin(r)}{\sqrt{1+(\lambda^2-1) \sin^2(r)}}.
\]
The family of transformations $\tau_{\lambda}$ satisfies $\tau_{\lambda\mu} = \tau_{\lambda}\circ \tau_{\mu}$ and, in our $(r,\phi)$ coordinates, is induced by the vector field 
\[
W=\left.\frac{\d}{\d\lambda}\right|_{\lambda=1} \tau_{\lambda} (r,\phi)=\frac{\sin(2r)}{2}\frac{\partial}{\partial r}.
\]
For our metric the functions $R^i$ are easily computed to be $R^0=R^2=0$ and
\begin{align*}
R^1&=\frac{(a_2-a_1)(\cos^2r+1)-2(a_1+a_2)\cos r}{((a_1-a_2)\cos r+a_1+a_2)\sin r},\\ 
R^3&=\frac{-2\sin(2r)}{((a_1-a_2)\cos r+a_1+a_2)^2}.
\end{align*}
It follows from elementary computations that the vector field $W$ does not solve the PDE system~\eqref{eq:pdeprojvec}. Therefore, the Lie algebra of projective vector fields of $(\mathbb{CP}(a_1,a_2),g)$ is spanned by the Killing vector field $\frac{\partial}{\partial \phi}$. 
\end{proof}


\begin{remark}
Alternatively, it is easy to check that $r(\phi)$-parametrizations of the curves $\tau_\lambda(\gamma)$ do not satisfy the geodesic equations \cite[4.1.12]{MR496885} for all $\lambda>0$, which also implies the claim. Also, a more refined but cumbersome analysis of the geodesic equations seems to show that $\tau_{\lambda}$ is only a projective transformation for $\lambda=1$, so that any projective transformation is in fact an isometry. 
\end{remark}

\begin{remark}
Note that if a connected group of projective transformations acts on a complete connected two-dimensional Riemannian manifold $(M,g)$, then it acts by isometries or $g$ has constant non-negative curvature~\cite{MR2165202} (see also~\cite{MR2301453} for the case of higher dimensions).  
\end{remark}

\subsection{Deformations of Finsler metrics and the Veronese embedding}\label{sub:deformations} Recall from \cref{subsec:twistor} that sections of $\mathsf{J}_+\to \mathbb{CP}(a_1,a_2)$ with holomorphic image correspond to Weyl connections in $[\nabla]$. Moreover, a projective transformation gives rise to a biholomorphism of $\mathsf{J}_+$, and it pulls-back a Weyl connection $\nabla_1\in [\nabla]$ to $\nabla_2 \in [\nabla]$ if and only if the corresponding holomorphic curves are mapped onto each other.

Let us identify $\mathsf{J}_+$ with $\mathbb{CP}(a_1,(a_1+a_2)/2,a_2)\setminus j(S^2)$ via  \cref{prp:weighted_biho} in the case where $\nabla=\nabla^g$ for the Besse orbifold metric $g$ from \cref{lem:fubinistudymetric}. In this case, the complex structure on $\mathbb{CP}(a_1,a_2)$ arising from the chosen orientation and the metric $g$ corresponds to the Veronese embedding
\[
		\begin{array}{cccl}
		 \Theta=\bihol_{|\Co\mathbb{P}(a_1,a_2)}: &\Co\mathbb{P}(a_1,a_2)		& \To  	&	\Co\mathbb{P}(a_1,(a_1+a_2)/2,a_2) \\
		           		& [z:w] 	& \MTo  & [z^2:zw:w^2].
		 \end{array}	
\]
We would like to construct deformations of Finsler metrics via deformations of this embedding. Note that the image of the Veronese embedding is defined by the equation $y^2_2=y_1y_3$, where we use $(y_1,y_2,y_3)$ as coordinates on $\mathbb{CP}(a_1,(a_1+a_2)/2,a_2)$. An explicit complex one-dimensional family of deformations is given by the equation $y_2^2=\lambda y_1y_3$ for some $\lambda \in \C^*$. Choosing $\lambda$ sufficiently close to $1$ will cut out a holomorphic curve which continues to be a section of $\mathsf{J}_+ \to \mathbb{CP}(a_1,a_2)$ and hence corresponds to a positive Weyl connection since the metric $g$ has strictly positive Gauss curvature. Therefore, according to our \cref{thm:duality}, small deformations of the Veronese embedding through holomorphic curves give rise to deformations of the Finsler metric dual to $g$ through Finsler metrics of constant curvature $1$ and all geodesics closed. 

It remains to show that the so obtained Finsler metrics are not all isometric. Again, according to our \cref{thm:duality}, this amounts to showing that the resulting Weyl structures do not coincide up to an orientation preserving diffeomorphism. Let $\mathscr{W}_{\lambda_i}$ for $i=1,2$ denote the Weyl structures corresponding to the deformations by $\lambda_1\neq\lambda_2$ sufficiently close to $1$. Suppose $\Psi : \mathbb{CP}(a_1,a_2) \to \mathbb{CP}(a_1,a_2)$ is an orientation preserving diffeomorphism which identifies $\mathscr{W}_{\lambda_1}$ with $\mathscr{W}_{\lambda_2}$. By construction, the Weyl structures $\mathscr{W}_{\lambda_i}$ have Weyl connections whose geodesics agree with the geodesics of the Besse orbifold metric $g$ up to parametrisation. Therefore, $\Psi$ is a projective transformation for the Levi-Civita connection of $g$ and hence by \cref{lem:groupprojtrans} an isometry for $g$ up to possibly applying a transformation from a discrete set of non-isometric projective transformations. 

Every orientation preserving diffeomorphism of $\mathbb{CP}(a_1,a_2)$ naturally lifts to a diffeomorphism of $\mathsf{J}_+$ and in the case of an orientation preserving isometry $\Upsilon : \mathbb{CP}(a_1,a_2) \to \mathbb{CP}(a_1,a_2)$ the lift $\mathsf{J}_+ \to \mathsf{J}_+$ is covered by a map $\mathrm{SU}(2)\times \mathbb{D} \to \mathrm{SU}(2)\times \mathbb{D}$ which is the product of the identity on the $\mathbb{D}$ factor and the natural lift of $\Upsilon$ to $\mathrm{SU}(2)$ on the first factor. With respect to our coordinates~\eqref{eq:spindlecoord} the isometries generated by the Killing vector field $\frac{\partial}{\partial \phi}$ lift to $\mathrm{SU}(2)$ to become left-multiplication by the element $\e^{\i\vartheta}$. Thus, under our biholomorphism $\Xi : \mathsf{J}_+ \to \mathbb{CP}(a_1,(a_1+a_2)/2,a_2)\setminus j(S^2)$ lifts of orientation preserving isometries to $\mathsf{J}_+$ take the form
\[
[y_1:y_2:y_3] \mapsto [\e^{-2\i\vartheta}y_1:y_2:\e^{2\i\vartheta}y_3]
\]
for $\vartheta \in \R$. Observe that each such transformation leaves each member of the family $y_2^2=\lambda y_1y_3$ invariant. In particular, the deformed Besse--Weyl structures are rotationally symmetric as well. Since $\Psi$ identifies the two Weyl structures, its lift $\tilde{\Psi} : \mathsf{J}_+ \to \mathsf{J}_+$ must map the holomorphic curves cut out by $y_2^2=\lambda_iy_1y_3$ for $i=1,2$ onto each other and hence $\Psi$ must be a member of the discrete set of non-isometric projective transformations. Since we have a real two-dimensional family of deformations of the Veronese embedding, we conclude that we have a corresponding real two-dimensional family of non-isometric, rotationally symmetric Finsler metrics of constant curvature $K=1$ on $S^2$ and with all geodesics closed. 

\begin{remark}
The Besse--Weyl structures arising from the deformations of the Veronese embedding are defined on $\mathbb{CP}(a_1,a_2)$ and hence on the Finsler side yield examples of Finsler $2$-spheres of constant curvature and with shortest closed geodesics of length $2\pi\left(\frac{a_1+a_2}{2a_1}\right)$.
\end{remark}

\begin{remark}\label{rem:projectively_equi_deformations}
To the best of our knowledge no two-dimensional family of deformations of rotationally symmetric Besse metrics on $\mathbb{CP}(a_1,a_2)$, $a_1,a_2>1$, in a fixed projective class is known.
\end{remark}

\appendix

\section{The Biholomorphism for the \texorpdfstring{$2$-Sphere}{2-sphere}}\label{appendix1}

Recall that the Killing form $B$ on $\mathfrak{su}(2)$ is negative definite. Therefore, fixing an isomorphism $(\mathfrak{su}(2),-B)\simeq \mathbb{E}^3$ with Euclidean $3$-space $\mathbb{E}^3$, the adjoint representation of $\SU$ gives a  group homomorphism
\[
\mathrm{Ad} : \SU \to \mathrm{SO}(\mathfrak{su}(2),-B)\simeq \mathrm{SO}(3). 
\]
Explicitly, mapping the $-B$-orthonormal basis of $\mathfrak{su}(2)$ given by
\[
b_1=\frac{\sqrt{2}}{4}\begin{pmatrix}0 & -1 \\ 1 & 0\end{pmatrix}\quad \text{and} \quad b_2=\frac{\sqrt{2}}{4}\begin{pmatrix}0 & \i \\ \i & 0\end{pmatrix}\quad \text{and} \quad b_3=\frac{\sqrt{2}}{4}\begin{pmatrix}-\i & 0 \\ 0 & \i\end{pmatrix}
\]
to the standard basis of $\R^3$, the adjoint representation becomes
\begin{multline*}
\esu \mapsto \\
\frac{1}{2}\begin{pmatrix} \left(z^2+w^2+\ov{z}^2+\ov{w}^2\right) & -\i\left(z^2+w^2-\ov{z}^2-\ov{w}^2\right) & 2\i(z\ov{w}-\ov{z}w) \\ \i\left(z^2-w^2-\ov{z}^2+\ov{w}^2\right) & \left(z^2+\ov{z}^2-w^2-\ov{w}^2\right) & -2(z\ov{w}+\ov{z}w) \\ 2\i(zw-\ov{zw}) & 2(zw+\ov{zw}) & 2(|z|^2-|w|^2)\end{pmatrix}.
\end{multline*}
The unit tangent bundle of the Euclidean $2$-sphere $S^2\subset \mathbb{E}^3$ may be identified with $\mathrm{SO}(3)$ by thinking of the third column vector $e$ of an element $(e_1\;e_2\;e)\in \mathrm{SO}(3)$ as the basepoint $e\in S^2$ and by thinking of the remaining two column vectors $(e_1\;e_2)$ as a positively oriented orthonormal basis of $T_{e}S^2$, where we identify the tangent plane to $e$ with the orthogonal complement $\{e\}^{\perp}$ of $\e$ in $\mathbb{E}^3$. 

We can represent an orientation compatible linear complex structure $J$ on $T_{e}S^2$ by mapping $J$ to the (complex) projectivisation of the vector $v+iJv$, where $v \in T_{e}S^2\simeq \{e\}^{\perp}\subset \R^3$ is any non-zero tangent vector. This defines a map $\mathsf{J}_+(S^2) \to \mathbb{CP}^2$. Since $v$ and $Jv$ are linearly independent, the image of this map is disjoint from $\mathbb{RP}^2$, where we think of $\mathbb{RP}^2$ as sitting inside $\mathbb{CP}^2$ via its standard real linear embedding. An elementary calculation shows that an orientation compatible linear complex structure $J$ on $T_{e}S^2$ with Beltrami coefficient $\mu \in \mathbb{D}$ relative to the standard complex structure $J_0$ has matrix representation
\[
J=\frac{1}{1-|\mu|^2}\begin{pmatrix}-2 \Im(\mu) & -|\mu|^2+2\Re(\mu)-1 \\  |\mu|^2+2\Re(\mu)+1 & 2\Im(\mu)\end{pmatrix}
\]
with respect to the basis $\{e_1,J_0e_1\}$ of $T_{e}S^2$. Hence, we obtain a map 
\begin{align*}
\mathrm{SO}(3) \times \mathbb{D} &\to \mathbb{CP}^2\setminus \mathbb{RP}^2 \\
\left[\left(e_1,e_2,e\right),\mu\right] &\mapsto \left[e_1+\frac{\i}{1-|\mu|^2}\left(-2\Im(\mu)e_1+(|\mu|^2+2\Re(\mu)+1)e_2\right)\right],
\end{align*}
where we have used that the standard complex structure $J_0$ of $S^2$ satisfies $e_2=J_0e_1$ for all $(e_1\;e_2\;e) \in \mathrm{SO}(3)$. Composing with the adjoint representation this becomes the map $\mathrm{SU}(2) \times \mathbb{D} \to \mathbb{CP}^2\setminus \mathbb{RP}^2$ defined by  
\begin{multline*}
\left[\left(z,w\right),\mu\right] \mapsto \\ \left[z^2+w^2-\mu(\ov{z}^2+\ov{w}^2):\i\left(z^2-w^2+\mu(\ov{z}^2-\ov{w}^2)\right):2\i(zw+\mu\ov{zw})\right]. 
\end{multline*}
A linear coordinate transformation 
\[
\C^3\to \C^3, \quad (z_1,z_2,z_3) \mapsto (\i z_1+ z_2,z_3,\i z_1-z_2)
\]
thus gives the map
\[
\SU\times \mathbb{D} \to \mathbb{CP}^2, \quad [(z,w),\mu] \mapsto [z^2-\mu \ov{w}^2 : zw+\mu\ov{zw}:w^2-\mu\ov{z}^2]
\]
used in \cref{subsec:twistorembedding}. 

\providecommand{\noopsort}[1]{}
\providecommand{\mr}[1]{\href{http://www.ams.org/mathscinet-getitem?mr=#1}{MR~#1}}
\providecommand{\zbl}[1]{\href{http://www.zentralblatt-math.org/zmath/en/search/?q=an:#1}{zbM~#1}}
\providecommand{\arxiv}[1]{\href{http://www.arxiv.org/abs/#1}{arXiv:#1}}
\providecommand{\doi}[1]{\href{http://dx.doi.org/#1}{DOI}}
\providecommand{\MR}{\relax\ifhmode\unskip\space\fi MR }
\providecommand{\MRhref}[2]{%
  \href{http://www.ams.org/mathscinet-getitem?mr=#1}{#2}
}
\providecommand{\href}[2]{#2}


\begin{thebibliography}{10}

\bibitem{MR2359514}
\bgroup\scshape{}A.~Adem\egroup{}, \bgroup\scshape{}J.~Leida\egroup{}, and
  \bgroup\scshape{}Y.~Ruan\egroup{}, \emph{Orbifolds and stringy topology},
  \emph{Cambridge Tracts in Mathematics} \textbf{171}, Cambridge University
  Press, Cambridge, 2007. \mr{2359514}\;

\bibitem{MR1052466}
\bgroup\scshape{}H.~Akbar-Zadeh\egroup{}, Sur les espaces de {F}insler \`a
  courbures sectionnelles constantes,  \emph{Acad.~Roy.~Belg.~Bull.~Cl.~Sci.
  (5)} \textbf{74} (1988), 281--322. \mr{1052466}\;

\bibitem{ALR}
\bgroup\scshape{}M.~Amann\egroup{}, \bgroup\scshape{}C.~Lange\egroup{}, and
  \bgroup\scshape{}M.~Radeschi\egroup{}, Odd-dimensional orbifolds with all
  geodesics closed are covered by manifolds, 2018. \arxiv{1811.10320}

\bibitem{MR496885}
\bgroup\scshape{}A.~L. Besse\egroup{}, \emph{Manifolds all of whose geodesics
  are closed}, \emph{Ergebnisse der Mathematik und ihrer Grenzgebiete [Results
  in Mathematics and Related Areas]} \textbf{93}, Springer-Verlag, Berlin-New
  York, 1978, With appendices by D. B. A. Epstein, J.-P. Bourguignon, L.
  B\'{e}rard-Bergery, M. Berger and J. L. Kazdan. \mr{496885}\;

\bibitem{MR1744486}
\bgroup\scshape{}M.~R. Bridson\egroup{} and
  \bgroup\scshape{}A.~Haefliger\egroup{}, \emph{Metric spaces of non-positive
  curvature}, \emph{Grundlehren der Mathematischen Wissenschaften [Fundamental
  Principles of Mathematical Sciences]} \textbf{319}, Springer-Verlag, Berlin,
  1999. \mr{1744486}\;

\bibitem{MR4195750}
\bgroup\scshape{}R.~L. Bryant\egroup{}, \bgroup\scshape{}P.~Foulon\egroup{},
  \bgroup\scshape{}S.~V. Ivanov\egroup{}, \bgroup\scshape{}V.~S.
  Matveev\egroup{}, and \bgroup\scshape{}W.~Ziller\egroup{}, Geodesic behavior
  for {F}insler metrics of constant positive flag curvature on {$S^2$},
  \emph{J. Differential Geom.} \textbf{117} (2021), 1--22. \mr{4195750}\;

\bibitem{bryantprescribedfinslercurvature}
\bgroup\scshape{}R.~L. Bryant\egroup{}, \emph{Finsler surfaces with prescribed
  curvature conditions},
  \href{https://services.math.duke.edu/~bryant/Finsler.dvi}{{U}npublished
  manuscript}, 1995.

\bibitem{MR1403574}
\bgroup\scshape{}R.~L. Bryant\egroup{}, Finsler structures on the {$2$}-sphere
  satisfying {$K=1$},  in \emph{Finsler geometry ({S}eattle, {WA}, 1995)},
  \emph{Contemp.~Math.} \textbf{196}, Amer. Math. Soc., Providence, RI, 1996,
  pp.~27--41. \mr{1403574}\;

\bibitem{Bry}
\bgroup\scshape{}R.~L. Bryant\egroup{}, Projectively flat {F}insler
  {$2$}-spheres of constant curvature,  \emph{Selecta Math.~(N.S.)} \textbf{3}
  (1997), 161--203. \mr{1466165}\;

\bibitem{MR1898190}
\bgroup\scshape{}R.~L. Bryant\egroup{}, Some remarks on {F}insler manifolds
  with constant flag curvature,  \emph{Houston J.~Math.} \textbf{28} (2002),
  221--262, Special issue for S. S. Chern. \mr{1898190}\;

\bibitem{MR2313331}
\bgroup\scshape{}R.~L. Bryant\egroup{}, Geodesically reversible {F}insler
  2-spheres of constant curvature,  in \emph{Inspired by {S}. {S}. {C}hern},
  \emph{Nankai Tracts Math.} \textbf{11}, World Sci. Publ., Hackensack, NJ,
  2006, pp.~95--111. \mr{2313331}\;

\bibitem{MR2368987}
\bgroup\scshape{}R.~L. Bryant\egroup{}, \bgroup\scshape{}G.~Manno\egroup{}, and
  \bgroup\scshape{}V.~S. Matveev\egroup{}, A solution of a problem of {S}ophus
  {L}ie: normal forms of two-dimensional metrics admitting two projective
  vector fields,  \emph{Math.~Ann.} \textbf{340} (2008), 437--463.
  \mr{2368987}\;

\bibitem{cartan1930}
\bgroup\scshape{}E.~Cartan\egroup{}, {Sur un probl\`eme d'\'equivalence et la
  th\'eorie des espaces m\'etriques g\'en\'eralis\'es.},  \emph{Mathematica}
  \textbf{4} (1930), 114--136 (French).

\bibitem{MR728412}
\bgroup\scshape{}M.~Dubois-Violette\egroup{}, Structures complexes au-dessus
  des vari\'{e}t\'{e}s, applications,  in \emph{Mathematics and physics
  ({P}aris, 1979/1982)}, \emph{Progr.~Math.} \textbf{37}, Birkh\"{a}user
  Boston, Boston, MA, 1983, pp.~1--42. \mr{728412}\;

\bibitem{MR0288785}
\bgroup\scshape{}D.~B.~A. Epstein\egroup{}, Periodic flows on three-manifolds,
  \emph{Ann.~of Math.~(2)} \textbf{95} (1972), 66--82. \mr{0288785}\;

\bibitem{MR1898191}
\bgroup\scshape{}P.~Foulon\egroup{}, Curvature and global rigidity in {F}insler
  manifolds,  \emph{Houston J.~Math.} \textbf{28} (2002), 263--292, Special
  issue for S. S. Chern. \mr{1898191}\;

\bibitem{MR3692897}
\bgroup\scshape{}H.~Geiges\egroup{} and
  \bgroup\scshape{}J.~Gonzalo~P\'{e}rez\egroup{}, Transversely holomorphic
  flows and contact circles on spherical 3-manifolds,  \emph{Enseign.~Math.}
  \textbf{62} (2016), 527--567. \mr{3692897}\;

\bibitem{GL16}
\bgroup\scshape{}H.~Geiges\egroup{} and \bgroup\scshape{}C.~Lange\egroup{},
  Seifert fibrations of lens spaces,  \emph{Abh.~Math.~Semin.~Univ.~Hambg.}
  \textbf{88} (2018), 1--22. \mr{3785783}\;

\bibitem{MR699802}
\bgroup\scshape{}N.~J. Hitchin\egroup{}, Complex manifolds and {E}instein's
  equations,  in \emph{Twistor geometry and nonlinear systems ({P}rimorsko,
  1980)}, \emph{Lecture Notes in Math.} \textbf{970}, Springer, Berlin-New
  York, 1982, pp.~73--99. \mr{699802}\;

\bibitem{MR1215481}
\bgroup\scshape{}Y.~Imayoshi\egroup{} and
  \bgroup\scshape{}M.~Taniguchi\egroup{}, \emph{An introduction to
  {T}eichm\"{u}ller spaces}, Springer-Verlag, Tokyo, 1992, Translated and
  revised from the Japanese by the authors. \mr{1215481}\;

\bibitem{MR3586335}
\bgroup\scshape{}T.~A. Ivey\egroup{} and \bgroup\scshape{}J.~M.
  Landsberg\egroup{}, \emph{Cartan for beginners}, \emph{Graduate Studies in
  Mathematics} \textbf{175}, American Mathematical Society, Providence, RI,
  2016, Second edition. \mr{3586335}\;

\bibitem{MR0331425}
\bgroup\scshape{}A.~B. Katok\egroup{}, Ergodic perturbations of degenerate
  integrable {H}amiltonian systems,  \emph{Izv.~Akad.~Nauk SSSR Ser.~Mat.}
  \textbf{37} (1973), 539--576. \mr{0331425}\;

\bibitem{MR1336823}
\bgroup\scshape{}S.~Kobayashi\egroup{}, \emph{Transformation groups in
  differential geometry}, \emph{Classics in Mathematics}, Springer-Verlag,
  Berlin, 1995, Reprint of the 1972 edition. \mr{1336823}\;

\bibitem{La16}
\bgroup\scshape{}C.~Lange\egroup{}, On metrics on 2-orbifolds all of whose
  geodesics are closed,  \emph{J. Reine Angew. Math.} \textbf{758} (2020),
  67--94. \mr{4048442}\;

\bibitem{La18}
\bgroup\scshape{}C.~Lange\egroup{}, Orbifolds from a metric viewpoint,
  \emph{Geom. Dedicata} \textbf{209} (2020), 43--57. \mr{4163391}\;

\bibitem{thesislebrun}
\bgroup\scshape{}C.~R. LeBrun\egroup{}, \emph{Spaces of complex geodesics and
  related structures}, Ph.D. thesis, University of Oxford, Bodleian Library
  Oxford, 1980, available at
  \url{https://ora.ox.ac.uk/objects/uuid:e29dd99c-0437-4956-8280-89dda76fa3f8}.

\bibitem{MR1979367}
\bgroup\scshape{}C.~R. LeBrun\egroup{} and \bgroup\scshape{}L.~J.
  Mason\egroup{}, Zoll manifolds and complex surfaces,  \emph{J.~Differential
  Geom.} \textbf{61} (2002), 453--535. \mr{1979367}\;

\bibitem{MR2747436}
\bgroup\scshape{}C.~R. LeBrun\egroup{} and \bgroup\scshape{}L.~J.
  Mason\egroup{}, Zoll metrics, branched covers, and holomorphic disks,
  \emph{Comm.~Anal.~Geom.} \textbf{18} (2010), 475--502. \mr{2747436}\;

\bibitem{MR2165202}
\bgroup\scshape{}V.~S. Matveev\egroup{}, Lichnerowicz-{O}bata conjecture in
  dimension two,  \emph{Comment.~Math.~Helv.} \textbf{80} (2005), 541--570.
  \mr{2165202}\;

\bibitem{MR2301453}
\bgroup\scshape{}V.~S. Matveev\egroup{}, Proof of the projective
  {L}ichnerowicz-{O}bata conjecture,  \emph{J.~Differential Geom.} \textbf{75}
  (2007), 459--502. \mr{2301453}\;

\bibitem{MR3144212}
\bgroup\scshape{}T.~Mettler\egroup{}, Weyl metrisability of two-dimensional
  projective structures,  \emph{Math.~Proc.~Cambridge Philos.~Soc.}
  \textbf{156} (2014), 99--113. \mr{3144212}\;

\bibitem{MR3384876}
\bgroup\scshape{}T.~Mettler\egroup{}, Geodesic rigidity of conformal
  connections on surfaces,  \emph{Math.~Z.} \textbf{281} (2015), 379--393.
  \mr{3384876}\;

\bibitem{MR3968880}
\bgroup\scshape{}T.~Mettler\egroup{} and \bgroup\scshape{}G.~P.
  Paternain\egroup{}, Holomorphic differentials, thermostats and {A}nosov
  flows,  \emph{Math. Ann.} \textbf{373} (2019), 553--580. \mr{3968880}\;

\bibitem{MR4109900}
\bgroup\scshape{}T.~Mettler\egroup{} and \bgroup\scshape{}G.~P.
  Paternain\egroup{}, Convex projective surfaces with compatible {W}eyl
  connection are hyperbolic,  \emph{Anal. PDE} \textbf{13} (2020), 1073--1097.
  \mr{4109900}\;

\bibitem{MR812312}
\bgroup\scshape{}N.~R. O'Brian\egroup{} and \bgroup\scshape{}J.~H.
  Rawnsley\egroup{}, Twistor spaces,  \emph{Ann. Global Anal. Geom.} \textbf{3}
  (1985), 29--58. \mr{812312}\;

\bibitem{MR0095520}
\bgroup\scshape{}I.~Satake\egroup{}, The {G}auss-{B}onnet theorem for
  {$V$}-manifolds,  \emph{J.~Math.~Soc.~Japan} \textbf{9} (1957), 464--492.
  \mr{0095520}\;

\bibitem{MR705527}
\bgroup\scshape{}P.~Scott\egroup{}, The geometries of {$3$}-manifolds,
  \emph{Bull.~London Math.~Soc.} \textbf{15} (1983), 401--487. \mr{705527}\;

\bibitem{MR1555366}
\bgroup\scshape{}H.~Seifert\egroup{}, Topologie {D}reidimensionaler
  {G}efaserter {R}\"{a}ume,  \emph{Acta Math.} \textbf{60} (1933), 147--238.
  \mr{1555366}\;

\bibitem{Thurston}
\bgroup\scshape{}W.~P. Thurston\egroup{}, The geometry and topology of
  three-manifolds, 1979, {L}ecture {N}otes, Princeton University.

\bibitem{MR1631588}
\bgroup\scshape{}S.~Zhu\egroup{}, Critical metrics on {$2$}-dimensional
  orbifolds,  \emph{Indiana Univ.~Math.~J.} \textbf{46} (1997), 1273--1288.
  \mr{1631588}\;

\bibitem{MR743032}
\bgroup\scshape{}W.~Ziller\egroup{}, Geometry of the {K}atok examples,
  \emph{Ergodic Theory Dynam.~Systems} \textbf{3} (1983), 135--157.
  \mr{743032}\;

\bibitem{ZOLL}
\bgroup\scshape{}O.~Zoll\egroup{}, Ueber {F}l\"{a}chen mit {S}charen
  geschlossener geod\"{a}tischer {L}inien,  \emph{Math.~Ann.} \textbf{57}
  (1903), 108--133. \mr{1511201}\;

\end{thebibliography}
\end{document}